\documentclass{amsart}

\usepackage{amsmath,
 amssymb,
 amsthm,
 bibentry,
 enumitem,
 float,
 graphicx,
 mathrsfs,
 tabu,
 tikz,
 xypic}
\usepackage[comma,numbers,sort,square]{natbib}
\usepackage[colorlinks=true,
 linktocpage=true,
 linkcolor=magenta,
 citecolor=magenta,
 urlcolor=magenta]{hyperref}
\PassOptionsToPackage{hyphens}{url}
\usetikzlibrary{arrows,shapes,snakes}

\newtheorem{theorem}{Theorem}[section]

\newtheorem{lemma}[theorem]{Lemma}
\newtheorem{corollary}[theorem]{Corollary}
\theoremstyle{definition}
\newtheorem{definition}[theorem]{Definition}
\newtheorem{statement}[theorem]{Statement}
\newtheorem{question}[theorem]{Question}

\newcommand{\res}{\upharpoonright}
\newcommand{\ran}{\operatorname{ran}}

\renewcommand{\mod}{\text{ }\textrm{mod}\text{ }}

\newcommand{\seq}[1]{\langle #1 \rangle}

\newcommand{\RCA}{\mathsf{RCA}}
\newcommand{\ACA}{\mathsf{ACA}}

\newcommand{\RT}{\mathsf{RT}}
\newcommand{\COH}{\mathsf{COH}}

\newcommand{\SRT}{\mathsf{SRT}}

\newcommand{\D}{\mathsf{D}}
\newcommand{\RADOMIL}{i}

\newcommand{\cred}{\leq_{\text{\upshape c}}}

\newcommand{\ncred}{\nleq_{\text{\upshape c}}}
\newcommand{\scred}{\leq_{\text{\upshape sc}}}
\newcommand{\nscred}{\nleq_{\text{\upshape sc}}}

\newcommand{\ured}{\leq_{\text{\upshape u}}}
\newcommand{\nured}{\nleq_{\text{\upshape u}}}

\newcommand{\sured}{\leq_{\text{\upshape su}}}
\newcommand{\nsured}{\nleq_{\text{\upshape su}}}

\begin{document}

\title{Strong reductions between combinatorial principles}

\author{Damir D. Dzhafarov}
\address{Department of Mathematics\\
University of Connecticut\\
196 Auditorium Road\\ Storrs, Connecticut 06269 U.S.A.}
\email{damir@math.uconn.edu}

\dedicatory{Dedicated to Robert I.\ Soare, on the occasion of his retirement.}

\thanks{The author was partially supported by NSF grant DMS-1400267, and is grateful to Denis Hirschfeldt for numerous helpful discussions.}

\maketitle

\begin{abstract}
This paper is a contribution to the growing investigation of strong reducibilities between $\Pi^1_2$ statements of second-order arithmetic, viewed as an extension of the traditional analysis of reverse mathematics. We answer several questions of Hirschfeldt and Jockusch~\cite{HJ-TA} about uniform and strong computable reductions between various combinatorial principles related to Ramsey's theorem for pairs. Among other results, we establish that the principle $\SRT^2_2$ is not uniformly or strongly computably reducible to $\D^2_{<\infty}$, that $\COH$ is not uniformly reducible to $\SRT^2_{<\infty}$, and that $\COH$ is not strongly reducible to $\SRT^2_2$. The latter also extends a prior result of Dzhafarov~\cite{Dzhafarov-2015}. We introduce a number of new techniques for controlling the combinatorial and computability-theoretic properties of the problems and solutions we construct in our arguments.
\end{abstract}

\section{Introduction}

The traditional approach in reverse mathematics has been to compare a given theorem with several benchmark subsystems of second-order arithmetic, and isolate the weakest that it can be proved in, and the strongest that can in turn be proved from it, all over the base theory $\RCA_0$. Often, these two subsystems coincide, meaning that the given theorem is actually equivalent over $\RCA_0$ to one of the benchmark subsystems. The early tendency in the subject was to focus almost exclusively on theorems of this kind, giving rise to what is often called the ``big five phenomenon'' (see, e.g., Montalb\'{a}n~\cite{Montalban-2011}, Section 1), after the five principal subsystems that arise most commonly in this endeavor. The past decade, however, has seen a growing shift away from this tendency, and towards looking at principles that do not admit equivalences to any of the benchmark subsystems. Spurred on primarily by interest in the strength of Ramsey's theorem for pairs, this new direction has since generated a zoo (see~\cite{Dzhafarov-zoo}) of inequivalent principles from various branches of mathematics with an intricate and fascinating structure of relationships. We refer the reader to Hirschfeldt's wonderful new text~\cite{Hirschfeldt-2014} for an introduction and summary. For a comprehensive general background on reverse mathematics, we refer to Simpson~\cite{Simpson-2009}.

\subsection{Strong reductions}

We take up an even newer direction in reverse mathematics, in which we replace provability over $\RCA_0$ by stronger relations, as a means of getting at finer distinctions between members of the above-mentioned zoo. The majority of principles studied in reverse mathematics are $\Pi^1_2$ statements of the form
\[
\forall X \Theta(X) \to \exists Y \Psi(X,Y),
\]
where $\Theta$ and $\Psi$ are arithmetical properties of sets. The sets $X$ satisfying $\Theta$ are then called the \emph{instances} of the principle, and the sets $Y$ satisfying $\Psi(X,Y)$ are the \emph{solutions} to $X$. To be precise, the above representation need not be unique, but for each principle we typically have a particular one in mind, and so with it, a particular set of instances and solutions. For example, in the statement of K\"{o}nig's lemma, that every finitely-branching tree of infinite height has an (infinite) path through it, the instances are the finitely-branching trees of infinite height, and the solutions the paths. There is a well-understood correspondence between the strength of a theorem, in the sense of reverse mathematics, and the relative complexities, in the sense of computability theory, of instances and solutions of that theorem
(see~\cite{HS-2007}, Section 1). For instance, if for every set $X$, principle $\mathsf{P}$ has an $X$-computable instance all of whose solutions compute $X'$, while principle $\mathsf{Q}$ satisfies that every $X$-computable instance of it has an $X'$-computable solution, then every $\omega$-model of $\mathsf{P}$ will also be a model of $\mathsf{Q}$. Thus, modulo issues of induction, a computability-theoretic fact yields an implication over $\RCA_0$.

More generally, consider the following reducibilities between $\Pi^1_2$ statements.

\begin{definition}\label{D:reductions}
	Let $\mathsf{P}$ and $\mathsf{Q}$ be $\Pi^1_2$ principles.
	\begin{enumerate}
		\item $\mathsf{Q}$ is \emph{computably reducible} to $\mathsf{P}$, written $\mathsf{Q} \cred \mathsf{P}$, if every instance $X$ of $\mathsf{P}$ computes an instance $\widehat{X}$ of $\mathsf{P}$, such that for every solution $\widehat{Y}$ to $\widehat{X}$, we have that $X \oplus \widehat{Y}$ computes a solution $Y$ to $X$.
		\item $\mathsf{Q}$ is \emph{strongly computably reducible} to $\mathsf{P}$, written $\mathsf{Q} \scred \mathsf{P}$, if every instance $X$ of $\mathsf{P}$ computes an instance $\widehat{X}$ of $\mathsf{P}$, such that every solution $\widehat{Y}$ to $\widehat{X}$ computes a solution $Y$ to $X$.
		\item $\mathsf{Q}$ is \emph{uniformly reducible} to $\mathsf{P}$, written $\mathsf{Q} \ured \mathsf{P}$, if there exist Turing functionals $\Phi$ and $\Psi$ such that for every instance $X$ of $\mathsf{P}$, we have that $\Phi^X$ is an instance of $\mathsf{P}$, and for every solution $\widehat{Y}$ to $\Phi^X$ we have that $\Psi^{X \oplus \widehat{Y}}$ is a solution to $X$.
		\item $\mathsf{Q}$ is \emph{strongly uniformly reducible} to $\mathsf{P}$, written $\mathsf{Q} \sured \mathsf{P}$, if there exist Turing functionals $\Phi$ and $\Psi$ such that for every instance $X$ of $\mathsf{P}$, we have that $\Phi^X$ is an instance of $\mathsf{P}$, and for every solution $\widehat{Y}$ to $\Phi^X$ we have that $\Psi^{\widehat{Y}}$ is a solution to $X$.
	\end{enumerate}
\end{definition}

\noindent The notions of $\scred$ and $\cred$ arise frequently in the proofs of implications over $\RCA_0$, but were first formally studied by~Dzhafarov\cite{Dzhafarov-2015}. For $\Pi^1_2$ principles, $\ured$ and $\sured$ are equivalent to Weihrauch reducibility and strong Weihrauch reducibility, respectively, which were introduced by Weihrauch~\cite{Weihrauch-1992} and have been widely applied in the study of computable analysis. (See~\cite{DDHMS-TA}, Appendix A, for a proof of the equivalences.) In the context of principles studied in reverse mathematics, these were defined independently by Doreais et al.~\cite{DDHMS-TA}.

It is easy to see that all of these reducibilities are transitive, and Figure~\ref{F:reductionsrelations} summarizes the relationships hold between them. Furthermore, if $\mathsf{Q}$ is reducible to $\mathsf{P}$ according to any one of these notions, then, as in the example above, every $\omega$-model of $\mathsf{P}$ is an $\omega$-model of $\mathsf{Q}$. Not every implication over $\RCA_0$, or even over $\omega$-models, must come from some such reduction, but in practice, most do. (Hirschfeldt and Jockusch~\cite{HJ-TA}, Section 4.1, introduced a related reduction notion that does capture implication over $\omega$-models.) Consequently, our motivation for studying principles under these reducibilities, as opposed to the traditional framework of reverse mathematics, is twofold. First, we can tease apart subtle differences between principles that are impossible to detect with provability in $\RCA_0$ alone. And second, where we do not know if there is an implication over $\RCA_0$ between two principles, we can lend credence to a negative answer by showing that some (or none) of the above stronger reducibilities hold. As a way of extending the scope of reverse mathematics, this program has been seeing increasing interest, as we describe further below.

\begin{figure}
\[
\xymatrix{
& \sured \ar[dl] \ar[dr] \\
\scred \ar[dr] & & \ured \ar[dl] \\
& \cred
}
\]
\caption[]{Relations between notions of reduction. An arrow from one reducibility to another means that whenever $\mathsf{Q}$ is reducible to $\mathsf{P}$ according to the first then it is also reducible according to the second. In general, no relations hold other than the ones shown.}\label{F:reductionsrelations}
\end{figure}
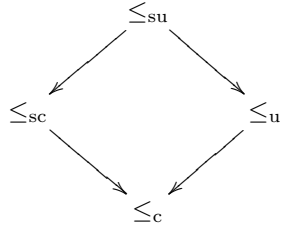

In the present article, we apply this analysis, for both of the above-mentioned motivations, to several combinatorial principles related to Ramsey's theorem for pairs. Our work is inspired by Questions~\ref{Q:HJ}~(1)--(4) below, posed by Hirschfeldt and Jockusch in a draft of~\cite{HJ-TA}, which we provide answers to here.


\subsection{Ramsey's theorem}

For a set $X$, let $[X]^n$ denote the set of all subsets $F$ of $X$ of size $n$. We identify each number $k \geq 1$ with its initial segment $\{0,\ldots,k-1\}$, and define a \emph{$k$-coloring of $[X]^n$} to be a map $c : [X]^n \to k$. As is customary, if $F \in [X]^n$ and the members of $F$ are $x_0 < \cdots < x_{n-1}$, we write $c(x_0,\ldots,x_{n-1})$ in place of $c(F)$. A set $H \subseteq X$ is \emph{homogeneous} for $c$ if there is a $j < k$ such that $c$ is constant on $[H]^n$ and equal to $j$, in which case we also say $H$ is homogeneous \emph{with color $j$}. Ramsey's theorem is the following principle. Typically, $X$ below will be $\omega$.

For all $n,k \geq 1$, we define the following $\Pi^1_2$ principle.

\begin{statement}[$\RT^n_k$]
	Every coloring $c : [\omega]^n \to k$ has an infinite homogeneous set.
\end{statement}

\noindent For any (standard) $k$ and $j$, it is easy to see that $\RT^n_k$ is equivalent over $\RCA_0$ to $\RT^n_j$. However, this result turns out to be sensitive to the distinctions drawn by Definition~\ref{D:reductions}. While $\RT^n_j \sured \RT^n_k$ for all $j < k$, it was shown by Dorais et al.~\cite[Theorem 3.1]{DDHMS-TA} that $\RT^n_k \nsured \RT^n_j$. This was subsequently improved by Hirschfeldt and Jockusch~\cite[Theorem 3.3]{HJ-TA} to show that also $\RT^n_k \nured \RT^n_j$, and more recently by Patey~\cite[Corollary 3.15]{Patey-TA} to show that even $\RT^n_k \ncred \RT^n_j$.

Note that $\RT^1_k$ is just the pigeonhole principle, and as such is provable in $\RCA_0$.  It follows by a well-known result of Jockusch~\cite[Lemma 5.9]{Jockusch-1972} that if $n \geq 3$ then $\RT^n_k$ is equivalent to $\ACA_0$ over $\RCA_0$. By contrast, Seetapun (see~\cite{SS-1995}, Theorem 2.1) showed that this is false for the $n =2$ case: $\RT^2_k$ is strictly weaker than $\ACA_0$. Seetapun's theorem set off an industry of research over the past twenty years into pinning down the precise strength of $\RT^2_k$. A prominent approach to this problem is that of Cholak, Jockusch, and Slaman~\cite[Section 7]{CJS-2001}, who introduced two closely related principles that $\RT^2_2$ can be decomposed into. The first is a restriction to a special class of colorings: say $c : [\omega]^2 \to k$ is \emph{stable} if for each $x$ the limit $\lim_y c(x,y)$ exists, meaning there is $j < k$ and a $z$ such that $c(x,y) = j$ for all $y \geq z$.

\begin{statement}[$\SRT^2_k$]
	Every stable coloring $c : [\omega]^2 \to k$ has an infinite homogeneous set.
\end{statement}

\noindent We let $\SRT^2_{<\infty}$ denote the statement $\forall k\,\SRT^2_k$ (and reserve the notation $\SRT^2$ for when we do not wish to specify a subscript). Typically, one works not with $\SRT^2_k$ directly, but with $\D^2_k$, defined as follows. Call a set $L$ \emph{limit-homogeneous} for a stable coloring $c : [\omega]^2 \to k$ if there is a $j < k$ such that $\lim_y c(x,y) = j$ for all $x \in L$. Thus, every homogeneous set is obviously limit-homogeneous. Conversely, if L is limit-homogeneous, then $c \oplus L$ can uniformly compute a homogeneous set by thinning.

\begin{statement}[$\D^2_k$]
	Every stable coloring $c : [\omega]^2 \to k$ has an infinite limit-homogeneous set.
\end{statement}

\noindent We let $\D^2_{<\infty}$ denote the statement $\forall k\,\D^2_k$. Using the limit lemma, it is easy to see that $\SRT^2_k$ and $\D^2_k$ are computably equivalent, as are $\SRT^2_{<\infty}$ and $\D^2_{<\infty}$ (see~\cite{CJS-2001}, Lemma 3.5, for a proof.) That this can be formalized in $\RCA_0$ is nontrivial, and was shown by Chong, Lempp, and Yang~\cite[Theorem 1.4]{CLY-2010}. As for $\RT$, there has been interest in $\SRT^2$ and $\D^2$ for different numbers of colors: Patey~\cite[Theorem 3.14]{Patey-TA} has shown that $\D^2_k \ncred \RT^2_j$ whenever $j < k$.

The second related principle introduced by Cholak, Jockusch, and Slaman~\cite{CJS-2001} is a sequential version of $\RT^1_2$ with finite errors. Given a family of sets $\vec{X} = \seq{X_n : n \in \omega}$, a set $Y$ is called \emph{$\vec{X}$-cohesive} if for each $n$, either $Y \cap X_n$ or $Y \cap \overline{X_n}$ if finite.

\begin{statement}[$\COH$]
	Every family of sets $\vec{X} = \seq{X_n : n \in \omega}$ has an infinite cohesive set $Y$.
\end{statement}

\noindent Over $\RCA_0$, the principles $\RT^2_k$ and $\SRT^2_k + \COH$ are equivalent (see Mileti~\cite{Mileti-2004}, Corollary A.1.4), and for many years it was a major open question in reverse mathematics whether this split is proper. Hirschfeldt et al.~\cite[Theorems 2.3 and 3.7]{HJKLS-2008} constructed an $\omega$-model of $\COH + \neg \SRT^2_2$, thereby showing that $\COH$ does not imply $\RT^2_2$, but the question of whether $\SRT^2_2$ implies $\COH$ was only answered considerably later, by Chong, Slaman, and Yang~\cite{CSY-TA}. The model for the the latter separation, however, has nonstandard first-order part, and so the question of whether every $\omega$-model of $\SRT^2_2$ is also a model of $\COH$ remains open.

In light of the discussion above, one way of moving towards a negative answer to this question is to compare the strengths of $\COH$ and $\SRT^2_k$ under the stronger reducibilities of Definition~\ref{D:reductions}. As a first step in this direction, Dzhafarov~\cite[Corollary 1.10]{Dzhafarov-2015} showed that $\COH \nscred \D^2_{<\infty}$. In an earlier draft of their own paper on stronger reducibilities between variants of Ramsey's theorem (since updated to reflect our results below), Hirschfeldt and Jockusch ~\cite{HJ-TA} asked about extensions of this fact.
\begin{question}[Hirschfeldt and Jockusch]\label{Q:HJ}
\
	\begin{enumerate}
		\item Is it the case that $\SRT^2_2 \ured \D^2_2$?
		\item Is it the case that $\SRT^2_2 \scred \D^2_2$?
		\item Is it the case that $\COH \ured \SRT^2_2$?
		\item Is it the case that $\COH \scred \SRT^2_2$?
	\end{enumerate}
\end{question}

\noindent Note that the fourth question is only of interest if the second has a negative answer, since otherwise its answer is no by the corresponding result for $\D^2_2$ mentioned above. The second in turn motivates the first, and the fourth motivates the third. There is an additional reason the first question is interesting. Namely, Hirschfeldt and Jockusch~\cite[Proposition 4.8]{HJ-TA} showed that the answer to it is almost yes, in that $\SRT^2_2$ is \emph{generalized} uniformly reducible to $\D^2_2$, namely to two iterated applications of $\D^2_2$ (see~\cite{HJ-TA}, Section 4.2, for precise definitions).

We give negative answers to all four of the above questions, and along the way obtain other results about these, and related, principles with regards to uniform and strong computable reducibility. In the case of each of the first three questions, our proofs will use an ad hoc combinatorial argument to diagonalize solutions to instances of the principle on the left, plus an intricate generalization of the proof of Seetapun's theorem from~\cite{SS-1995} to build solutions to instances of the principle on the right. Consequently, we dedicate Section~\ref{S:gencore} to developing this generalization. By contrast, to obtain a negative answer to the fourth question we introduce a new method for directly constructing homogeneous sets for colorings. While the arguments based on our extension of Seetapun's argument produce relatively effective constructions, the same is not true of this new method, which uses countably many iterates of the hyperjump.

Our computability-theoretic notation and conventions throughout will be standard, following Soare~\cite{Soare-TA} or Downey and Hirschfeldt~\cite{DH-2010}.

\section{A generalization of Seetapun's argument}\label{S:gencore}

For a tree $T \subseteq \omega^{< \omega}$, we write $\ran(T)$ for $\bigcup_{\alpha \in T} \ran(\alpha)$, and $|T|$ for $\sup_{\alpha \in T} |\alpha|$. For an infinite set $X$, let $Inc(X)$ be the set of all increasing sequences of elements of $X$, i.e., all $\sigma \in \omega^{<\omega}$ with $\sigma(n) \in X$ for all $n < |\sigma|$ and $\sigma(n) < \sigma(n')$ for all $n < n' < |\sigma|$.
So if $T$ is a subtree of $Inc(X)$, then any path through $T$ has infinite range. Additionally, if such a $T$ has bounded width, then it is finite as a set of strings if and only $\ran(T)$ is finite, if and only if $|T|$ is a finite number, if and only if $T$ is well-founded.

Throughout, if $X$ and $Y$ are sets, we will write $X < Y$ to mean that $X$ is finite and $\max X < \min Y$. If $A = \{x_0,\ldots,x_{n-1}\}$, we will also sometimes write $x_0,\ldots,x_{n-1} < Y$. If $T$ and $U$ are subtrees of $Inc(\omega)$, we will write $T < U$ to mean that $\ran(T) < \ran(U)$. Note that this implies that $\ran(T)$ is finite, hence in particular that $T$ is well-founded.

\begin{definition}\label{D:phisequences}
Let $\varphi$ and $\psi$ be properties of finite sets.
\begin{enumerate}
\item A \emph{$\varphi$-tree} is a finite subtree $T$ of $Inc(\omega)$ of bounded width such that if $\alpha \in T$ is a terminal node, then $\varphi(F)$ holds for some finite set $F \subseteq \ran(\alpha)$.
\item A \emph{$\varphi$-sequence} is a (finite or infinite) sequence $T_0 < T_1 < \cdots$ of $\varphi$-trees.
\item The \emph{$\psi$-generated subtree} of a $\varphi$-sequence $T_0 < T_1 < \cdots$ is the tree of all $\alpha \in \omega^{<\omega}$ such that $\alpha(n) \in \ran(T_n)$ for all $n < |\alpha|$, and $\neg \psi(F)$ holds for all finite sets $F \subseteq \ran(\alpha \res |\alpha|-1)$.
\end{enumerate}
\end{definition}

\noindent We highlight a few basic but important observations about the above definition.
\begin{itemize}
\item The $\psi$-generated subtree is a finitely-branching subtree of $Inc(\omega)$.
\item If the $\psi$-generated subtree of an infinite $\varphi$-sequence is finite, it is a $\psi$-tree.
\item If the $\psi$-generated subtree $U$ of some $\varphi$-sequence $T_0 < T_1 < \cdots$ is a $\psi$-tree, then $U$ is also the $\psi$-generated subtree of $T_0 < \cdots < T_{|U|-1}$.
\item If $\varphi$ is $\Sigma^0_1$ and there exists an infinite $\varphi$-sequence, then there also exists a computable infinite $\varphi$-sequence.
\item If $\psi$ is $\Sigma^0_1$, then the $\psi$-generated subtree of any computable infinite $\varphi$-sequence is computable and computably-bounded.
\end{itemize}

\noindent We add to this list the following two general results.

\begin{lemma}\label{L:finseq}
If there exists no infinite $\varphi$-sequence, then there is a number $z$ such that $\neg \varphi(F)$ holds for all finite sets $F > z$.
\end{lemma}

\begin{proof}
If there is no such $z$, we can define a sequence of finite sets $F_0 < F_1 < \cdots$ such that $\varphi(F_n)$ holds for all $n$. For each $n$, let $T_n$ be the tree of all initial segments of the principal function of $F_n$. Then $T_0 < T_1 < \cdots$ is an infinite $\varphi$-sequence.
\end{proof}

\begin{lemma}\label{L:twoseq}
If there exists an infinite $\varphi$-sequence, and the $\psi$-generated subtree of every infinite $\varphi$-sequence is finite, then there exists an infinite $\psi$-sequence.
\end{lemma}

\begin{proof}
Let $T_0 < T_1 < \cdots$ be any $\varphi$-sequence, and assume that for some $n \in \omega$, we have defined $\psi$-trees $U_m$ for all $m < n$. Since all the $U_m$ are finite, there is an $s$ such that $U_m < T_s$ for all $m < n$. Now $T_s < T_{s+1} < \cdots$ is a $\varphi$-sequence, so if we let $U_n$ be its $\psi$-generated subtree, then $U_n$ is finite and hence a $\psi$-tree. By choice of $s$, we have $U_m < U_n$ for all $m < n$, so by induction, $U_0 < U_1 < \cdots$ is a $\psi$-sequence.
\end{proof}

It is easy to see that if we chose $s$ at each step of the above construction to be minimal, we would end up making $s = \sum_{m < n} |T_{j+1,m}|$.


\begin{definition}\label{D:forest}
Fix $k \geq 1$, and let $\varphi_0,\ldots,\varphi_{k-1}$ be properties of finite sets. A \emph{$\seq{\varphi_0,\ldots,\varphi_{k-1}}$-forest} is a collection of sequences of trees
\[
\{T_{0,j} < \cdots < T_{j,s_j} : j < k\}
\]
with the following properties. For each $j < k$, the sequence $T_{j,0} < \cdots < T_{j,s_j}$ is a $\varphi_j$-sequence, and for each $j < k-1$ and each $n < s_{j+1}$, the tree $T_{j+1,n}$ is the $\psi_{j+1}$-generated subtree of
\[
T_{j,s} < \cdots <T_{j,s + |T_{j+1,n}|-1},
\]
where $s = \sum_{m < n} |T_{j+1,m}|$.
\end{definition}

\noindent So for instance, in a $\seq{\varphi_0,\varphi_1,\varphi_2}$-forest $\{T_{0,j} < \cdots < T_{j,s_j} : j < 3\}$, we have:
\begin{itemize}
	\item the tree $T_{1,0}$ is the $\varphi_1$-generated subtree of $T_{0,0} < \cdots < T_{0,|T_{1,0}|-1}$;
	\item the tree $T_{1,1}$ is the $\varphi_1$-generated subtree of $T_{0,|T_{1,0}|} < \cdots < T_{0,|T_{1,0}| + |T_{1,1}| - 1}$;
	\item the tree $T_{2,0}$ is the $\varphi_2$-generated subtree of $T_{1,0} < \cdots < T_{1,|T_{2,0}|-1}$;
\end{itemize}
and so on.

\begin{lemma}\label{L:exforest}
Fix $k \geq 1$, let $\varphi_0,\ldots,\varphi_{k-1}$ be properties of finite sets. Suppose there exists an infinite $\varphi_0$-sequence, and that for each $j < k - 1$, every infinite $\varphi_j$-sequence has finite $\varphi_{j+1}$-generated subtree. Then there exists a $\seq{\varphi_0,\ldots,\varphi_{k-1}}$-forest, and moreover, if $\varphi_0,\ldots,\varphi_{k-1}$ are $\Sigma^0_1$, then it can be found uniformly computably.
\end{lemma}

\begin{proof}
By Lemma~\ref{L:twoseq}, there exists an infinite $\varphi_j$-sequence $T_{j,0} < T_{j,1} < \cdots$ for each $j < k$. By the proof of that lemma and the subsequent remark, we may choose these sequences so that for each $j < k-1$ and every $n$, the tree $T_{j+1,n}$ is the $\varphi_{j+1}$-generated subtree of the sequence $T_{j,s} < T_{j,s+1} < \cdots$, where $s = \sum_{m < n} |T_{j+1,m}|$. Let $s_{k-1} = 0$, and for $j < k-1$, let $s_j = (\sum_{n \leq s_{j+1}} |T_{j+1,n}|) - 1$. Then $\{T_{0,j} < \cdots < T_{j,s_j} : j < k\}$ is a $\seq{\varphi_0,\ldots,\varphi_{k-1}}$-forest.
\end{proof}

The next result is the main motivation behind Definition~\ref{D:phisequences}, and appears in less general form in~\cite[Lemma 2.3]{Dzhafarov-2015}. It is an elaboration on the combinatorial core of Seetapun's argument~\cite[Lemma 2.14]{SS-1995}, which is essentially the special case when $k = 2$ below.

\begin{lemma}\label{L:combcore}
Fix $k \geq 1$, let $\varphi_0,\ldots,\varphi_{k-1}$ be properties of finite sets, and suppose $\{T_{0,j} < \cdots < T_{j,s_j} : j < k\}$ is a $\seq{\varphi_0,\ldots,\varphi_{k-1}}$-forest. Given $l : \omega \to k$, there is a $j < k$ and a terminal $\alpha$ in some $T_{j,n}$ such that $l(x) = j$ for all $x \in \ran(\alpha)$.
\end{lemma}

\begin{proof}
Suppose the conclusion of the lemma fails for all $j < k-1$. We claim that for each such $j$, every $T_{j,n}$ contains some terminal $\alpha$ having an $x$ in its range with $l(x) > j$. This is true of $j = 0$ by assumption, so assume it for an arbitrary $j < k-1$. Fix any $n \leq s_{j+1}$, and let $s = \sum_{m < n} |T_{j+1,m}|$ and $t = s + |T_{j+1,n}|-1$, so that $T_{j+1,n}$ is the $\varphi_{j+1}$-generated subtree of $T_{j,s} < \cdots < T_{j,t}$. For each $m$ with $s \leq m \leq t$, choose an $x_m$ in some $\alpha \in T_{j,m}$ with $l(x_m) > j$. Then $\alpha = \seq{x_s,\ldots,x_t}$ is a terminal node in $T_{j+1,n}$, and we have $l(x) \geq j+1$ for all $x \in \ran(\alpha)$. If $j+1 < k-1$, we must have that $l(x) > j+1$ for some $x \in \ran(\alpha)$, by assumption, which completes the induction. If, on the other hand, $j+1 = k-1$, then we have $l(x) = k-1$ for all $x \in \ran(\alpha)$. Thus, the conclusion of the lemma holds for $k-1$.
\end{proof}

We close by noting that Definitions~\ref{D:phisequences} and~\ref{D:forest} can both be formulated for restrictions to any infinite set $X$, by replacing $Inc(\omega)$ with $Inc(X)$ throughout. For instance, we define a \emph{$\varphi$-forest inside $X$} to be a finite subtree $T$ of $Inc(X)$ of bounded width such that if $\alpha \in T$ is terminal then $\varphi(F)$ holds for some finite $F \subseteq \ran(\alpha)$. All of the above lemmas extend in the obvious way to this more general setting.

\section{Comparing $\SRT^2$ and $\D^2$}\label{S:stable}

In this section, we prove that $\SRT^2_2$ is not uniformly or strongly computably reducible to $\D^2_{<\infty}$. The first of these results is comparatively straightforward, and thus serves as a good illustration of how the setup of the preceding section will be applied, both in the proof of the second result, as well as in the considerably more complicated arguments of Section~\ref{S:COH_uniform}.

We begin with the following definition.

\begin{definition}
	Let $\mathbb{P}$ be the following notion of forcing. A \emph{condition} is a triple $p = \seq{n^p,c^p,\ell^p}$, where $n^p$ is a number, $c^p$ is a function $[n^p]^2 \to 2$, and $\ell^p$ is a function $n^p \to 2 \times \omega$ such that if $\ell^p(x) = \seq{v,u}$ for some $x < n^p$ then $c^p(x,y) = v$ for all $y$ with $u \leq y < n^p$. A condition $q$ \emph{extends} $p$ if $n^q \geq n^p$, $c^q \supseteq c^p$, and $\ell^q \supseteq \ell^p$.
\end{definition}

\noindent Thus, $c^p$ is a finite approximation to a coloring of pairs, and we think of $\ell^p$ as a commitment about what color the numbers smaller than $n^p$ should limit to under this coloring, and by what point their colors should stabilize. It is clear that if $G$ is a sufficiently generic filter on $\mathbb{P}$, then $\bigcup_{p \in G} c^p$ is a stable coloring of pairs. We say a stable coloring $c : [\omega]^2 \to 2$ \emph{extends} a condition $p$, and write $c \succeq p$, if $c$ extends $c^p$ as a function and respects $\ell^p$, in the sense that if $\ell^p(x) = \seq{v,u}$ then $c(x,y) = v$ for all $y \geq u$.

Our results below, while driven by an interest in $\ured$ and $\scred$, actually prove considerably more. Namely, each constructs an instance $c$ of $\SRT^2_2$ such that \emph{every} instance of $\D^2_{<\infty}$, computable from $c$ or not, has a solution not computing (uniformly or strongly, respectively) any solution to $c$. Such stronger versions of $\ured$ and $\scred$ were first asked about by Montalb\'{a}n (unpublished), in connection with the results of~\cite{Dzhafarov-2015}, and have subsequently also emerged in~\cite{HJ-TA} and~\cite{Patey-TA}. The ``forward'' reductions in the definitions of these reductions, in other words, seem less essential in many cases.

\begin{theorem}\label{T:nured_no}
	Let $\Psi$ be a Turing functional. There is a computable coloring $c : [\omega]^2 \to 2$ such that for every $k \geq 2$, every stable coloring $d : [\omega]^2 \to k$ has an infinite limit-homogeneous set $L$ with the property that $\Psi^L$ is not equal to any infinite homogeneous set for $c$.
\end{theorem}

\begin{proof}
	We may assume $\Psi$ is $\{0,1\}$-valued, so that if it is total, it defines a set. Also, if there is an infinite set $X$ such that $\Psi^X$ does not define an infinite set for any $Y \subseteq X$, we may let $c$ be arbitrary, and for each $d : [\omega]^2 \to k$ simply choose an infinite limit-homogeneous set $L \subseteq X$. Thus, assume this is not the case. We build a computable sequence of conditions $p_0 \geq p_1 \geq \cdots$, and let $c = \bigcup_{k \in \omega} c^{p_k}$.
	
	Let $p_0 = \seq{0,\emptyset,\emptyset}$, and assume inductively that we have defined $p_k$ for some $k \in \omega$. Let $\varphi(F)$ be the formula asserting that there exist two numbers $x_1 > x_0 \geq n^{p_k}$ with $\Psi^{F}(x_0) \downarrow = \Psi^{F}(x_1) \downarrow = 1$. We consider the following three cases.
	
	Case 1: there exists no infinite $\varphi$-sequence. By Lemma~\ref{L:finseq}, there is a number $z$ such that $\neg \varphi(F)$ holds for all finite sets $F > z$. In this case, if $Y$ is any subset of $\{x \in \omega : x > z\}$ and $\Psi^Y$ is total, then the set it defines contains at most one number $x \geq n^{p_k}$, and so is finite. This contradicts our assumption above about $\Psi$.
	
	Case 2: there exists an infinite $\varphi$-sequence with infinite $\varphi$-generated subtree. Let $P$ be any path through this tree, recalling that it has infinite range. By definition, $\neg \varphi(F)$ holds for every finite set $F \subseteq \ran(P)$. Hence, if $Y$ is any subset of $\ran(P)$ and $\Psi^Y$ is total, the set it defines is finite, which is again a contradiction.
	
	Case 3: otherwise. Define $\varphi_0 = \cdots = \varphi_{k-1} = \varphi$, and observe that by the failure of Cases 1 and 2, we are exactly in the hypotheses of Lemma~\ref{L:exforest}. Since $\varphi$ is $\Sigma^0_1$, we may consequently effectively search for, and find, a $\seq{\varphi_0,\ldots,\varphi_{k-1}}$-forest $\seq{T_{j,0} < \cdots < T_{j,s_j} : j < k }$. Since each $T_{j,i}$ is a $\varphi_j$-tree, for every terminal $\alpha \in T_{j,i}$ there is a finite set $F$ and two numbers $x_1 > x_0 \geq n^{p_k}$ with $\Psi^{F}(x_0) \downarrow = \Psi^{F}(x_1) \downarrow = 1$. Let $u$ be the maximum of all these computations, across all $\alpha$ and $F$, which by usual conventions also means that $u$ is bigger than all the numbers $x_0$ and $x_1$. Let $p_{k+1}$ be the extension of $p_k$ with $n^{p_{k+1}} = u$, and with $c^{p_{k+1}}(x,y) = 0$ and $\ell^{p_{k+1}}(x) = \seq{1,u}$ for all $x,y$ with $n^{p_k} \leq x < y < u$.
	
	Now let $d : [\omega]^2 \to k$ be given, and by passing to a smaller number of colors if necessary, assume that for each $j < k$, there are infinitely many $x$ with $\lim_y d(x,y) = j$. Let $\seq{T_{j,0} < \cdots < T_{j,s_j} : j < k }$ be the $\seq{\varphi_0,\ldots,\varphi_{k-1}}$-forest from the construction of $p_{k+1}$. By Lemma~\ref{L:combcore}, there is a $j < k$, an $i \leq s_j$, and a terminal $\alpha \in T_{j,i}$, such that $\lim_y d(x,y) = j$ for all $x \in \ran(\alpha)$. Fix $F \subseteq \ran(\alpha)$ and $x_1 > x_0 \geq n^{p_k}$ with $\Psi^F(x_0) \downarrow = \Psi^F(x_1) \downarrow = 1$. Since infinitely many numbers limit to $j$ under $d$, there is an infinite limit-homogeneous set $L$ for $d$ extending $F$ above the use of these computations. In particular, $\Psi^L$ agrees with $\Psi^F$ on $x_0$ and $x_1$, so if $\Psi^L$ is total the set it defines contains $x_0$ and $x_1$. But by construction, $c(x_0,x_1) = 0$ and $\lim_y c(x_0) = 1$, so this set cannot be homogeneous for $c$.
\end{proof}

\begin{corollary}\label{C:SRT22nuredD2}
	$\SRT^2_2 \nured \D^2_{<\infty}$.
\end{corollary}

For our next result, recall that a set $X$ is said to have \emph{PA degree} relative to a set $Y$, written $X \gg Y$, if every $Y$-computable infinite finitely-branching tree has an $X$-computable path through it. By a result of Simpson~\cite[Theorem 6.5]{Simpson-1977}, if $X \gg Y$ then there is a $Z$ with $X \gg Z \gg Y$. Below, we say a coloring $c : [\omega]^2 \to 2$ is \emph{$1$-generic} if for each $\Sigma^0_1$-definable class of $\mathbb{P}$-conditions $W$ there is a $p \preceq c$ that meets $W$ (meaning $p \in W$) or avoids it (meaning $q \notin W$ for all $q \leq p$).

\begin{theorem}\label{T:scred_no}
	Let $X$ be a set of PA degree, and let $c : [\omega]^2 \to 2$ be $1$-$X$-generic for $\mathbb{P}$.	For every $k \geq 2$ and every infinite set $Y \ll X$, every stable coloring $d : [Y]^2 \to k$ has an infinite limit-homogeneous set $L \subseteq Y$ that computes no infinite homogeneous set for $c$.
\end{theorem}

\begin{proof}
	Fix $k$, $Y$, and $d : [Y]^2 \to k$. By passing to a smaller number of colors and a subset of $Y$ if necessary, we may assume that for every infinite $Z \subseteq Y$ satisfying $Z \ll X$, and for each $j < k$, there are infinitely many $x \in Z$ with $\lim_y d(x,y) = j$. Note that we can do this and still assume that $k \geq 2$, as otherwise $d$ would have an $Y$-computable infinite limit-homogeneous set, while by $1$-genericity relative to $X$, no $X$-computable set, and hence no $Y$-computable set, can be infinite and homogeneous for $c$.
	
	We build infinite limit-homogeneous sets $L_0,\ldots,L_{k-1}$ for $d$, with colors $0,\ldots,k-1$, respectively, such that for every functional $\Psi$ there is a $j < k$ with the property that $\Psi^{L_j}$ is not equal to any infinite homogeneous set for $c$. From here it follows that one of the $L_j$ computes no such homogeneous set, as desired.

	We build $L_0,\ldots,L_{k-1}$ by forcing with conditions $(F_0,\ldots,F_{k-1},I)$, where each $F_j$ is a finite set satisfying $\lim_y d(x,y) = j$ for all $x \in F$, and $I > F_0,\ldots,F_{k-1}$ is an infinite subset of $Y$ satisfying $I \ll X$. A condition $(F_0',\ldots,F_{k-1}',I')$ extends $(F_0,\ldots,F_{k-1},I)$ if $F_j \subseteq F_j' \subseteq F_j \cup I$ for each $j < k$, and $I' \subseteq I$. Thus, this is just forcing with $k$ many Mathias conditions sharing a common reservoir. We define
	\[
		(F_{0,0},\ldots,F_{0,k-1},I_0) \geq (F_{1,0},\ldots,F_{1,k-1},I_1) \geq \cdots
	\]
	with $\lim_s |F_{s,j}| = \infty$ for all $j < k$, and take $L_j = \bigcup_s F_{s,j}$. Let $(F_{0,0},\ldots,F_{0,k-1},I_0) = (\emptyset,\ldots,\emptyset,\omega)$, and suppose that we have defined $(F_{s,0},\ldots,F_{s,k-1},I_s)$ for some $s$.

	If $s$ is even, we wish to add one more element to each of the eventual sets $L_j$. By our assumption above, for each $j < k$, there must be infinitely many $x \in I_s$ with $\lim_y d(x,y) = j$. Hence, we may define $(F_{s+1,0},\ldots,F_{s+1,k-1},I_{s+1})$ so that $|F_{s+1,j}| = |F_{s,j}| + 1$ for all $j$, as desired.

	Now let $s$ be odd, say $s = 2e + 1$. Let $\Psi$ be the $e$th member of some fixed listing of all Turing functionals. For each $m \geq \min I_s$ and each $j < k$, let $\varphi^m_j(F)$ be the formula asserting that $F \geq m$ and there exist two numbers $x_1 > x_0 \geq m$ with $\Psi^{F_{s,j} \cup F}(x_0) \downarrow = \Psi^{F_{s,j} \cup F}(x_1) \downarrow = 1$. Let $D$ be the set of all $\mathbb{P}$-conditions $p$ such that for some $m \geq n^p$, the following are true.
	\begin{itemize}
		\item There exists a $\seq{\varphi^m_0,\ldots,\varphi^m_{k-1}}$-forest $\seq{T_{j,0} < \cdots < T_{j,s_j} : j < k }$ inside $I_s$.
		\item For each terminal $\alpha$ in each $T_{j,i}$, there is a finite set $F \geq m$ contained in $\ran(\alpha)$ and two numbers $x_1 > x_0 \geq m$ with $\Psi^{F_{s,j} \cup F}(x_0) \downarrow = \Psi^{F_{s,j} \cup F}(x_1) \downarrow = 1$, such that $x_0,x_1 < n^p$ and $c^p(x_0,x_1) = 0$ and $\ell^p(x_0) = \seq{1,u}$ for some $u$.
	\end{itemize}
	Then $D$ is $\Sigma^0_1$-definable over $I_s$, and hence over $X$, so since $c$ is $1$-$X$-generic, there must be a condition $p \preceq c$ that meets or avoids $D$.

	Suppose first that $p$ meets $D$. By Lemma~\ref{L:combcore}, there is a $j < k$, an $i \leq s_j$, and a terminal $\alpha \in T_{j,i}$, such that $\lim_y d(x,y) = j$ for all $x \in \ran(\alpha)$. Since $m \geq \min I_s$, by definition of $D$ there is an $F \geq \min I_s$ contained in $\ran(\alpha)$ and $x_1 > x_0 \geq \min I_s$ with $\Psi^{F_{s,j} \cup F}(x_0) \downarrow = \Psi^{F_{s,j} \cup F}(x_1) \downarrow = 1$, such that $c^p(x_0,x_1) = 0$ and $\ell^p(x_0) = 1$. In particular, $\lim_y d(x,y) = j$ for all $x \in F$, and since $T_{j,i}$ is a $\varphi^m_j$-tree inside $I_s$, also $F \subseteq I_s$. Let $F_{s+1,j} = F_{s,j} \cup F$, and for all $j' \neq j$ let $F_{s+1,j'} = F_{s,j'}$. Let $I_{s+1}$ be the set of all elements in $I_s$ greater than $F$ and the uses of the computations $\Psi^{F_{s,j} \cup F}(x_0)$ and $\Psi^{F_{s,j} \cup F}(x_1)$. Then $(F_{s+1,0},\ldots,F_{s+1,k-1},I_{s+1})$ is an extension of $(F_{s,0},\ldots,F_{s,k-1},I_s)$ guaranteeing that if $\Psi^{L_j}$ is total, it contains $x_0$ and $x_1$. But as $c(x_0,x_1) = 0$ and $\lim_y c(x_0,y) = 1$, this set is not homogeneous for $c$.

	So suppose instead that $p$ avoids $D$, and let $m$ be the larger of $\min I_s$ and $n^p$. We consider the following three cases.

	Case 1: for some $j < k$, there exists no infinite $\varphi^m_j$-sequence inside $I_s$. Fix $n_0$ such that $\neg \varphi(F)$ holds for all finite sets subsets $F > n_0$ of $I$. Let $I_{s+1} = \{x \in I_s : x > n_0\}$ and let $F_{s+1,0},\ldots,F_{s+1,k-1}$ be $F_{s,0},\ldots,F_{s,k-1}$, respectively. Then $(F_{s+1,0},\ldots,F_{s+1,k-1},I_{s+1})$ extends $(F_{s,0},\ldots,F_{s,k-1},I_s)$, and if $B$ is any set with $F_{s+1,j} \subseteq B \subseteq F_{s+1,j} \cup I_{s+1}$ and $\Psi^B$ is total, the set it defines can contain at most one numbers $x \geq m$. In particular, $\Psi^{L_j}$ cannot be an infinite set.

	Case 2: for some $j < k-1$, there exists an $I_s$-computable infinite $\varphi^m_j$-sequence inside $I_s$ with infinite $\varphi^m_{j+1}$-generated subtree. Notice that every path through this subtree is an infinite subset of $I_s$. Moreover, this subtree is $I_s$-computable and $I_s$-computably bounded, so since $I_s \ll X$, there is a path $P$ satisfying $P \ll X$. By definition, $\neg \varphi(F)$ holds for every finite set $F \subseteq \ran(P)$. Let $I_{s+1} = \ran(P)$ and let $F_{s+1,0},\ldots,F_{s+1,k-1}$ be $F_{s,0},\ldots,F_{s,k-1}$, respectively. As in the previous case, this ensures that $\Psi^{L_{j+1}}$ cannot be an infinite set.

	Case 3: otherwise. We are in the hypotheses of Lemma~\ref{L:exforest}, so we may fix a $\seq{\varphi^m_0,\ldots,\varphi^m_{k-1}}$-forest $\seq{T_{j,0} < \cdots < T_{j,s_j} : j < k }$ inside $I_s$. For every terminal $\alpha \in T_{j,i}$ there is a finite set $F$ and two numbers $x_1 > x_0 \geq m$ with $\Psi^{F_{s,j} \cup F}(x_0) \downarrow = \Psi^{F_{s,j} \cup F}(x_1) \downarrow = 1$. Let $u$ be the maximum of all these computations, across all $\alpha$ and $F$, which means $u$ is also bigger than all the numbers $x_0$ and $x_1$. Let $q$ be the extension of $p$ with $n^{q} = u$, and with $c^{q}(x,y) = 0$ and $\ell^{q}(x) = \seq{1,u}$ for all $x,y$ with $n^p \leq x < y < u$. In particular, since $m \geq n^p$ and $u > x_1 > x_0 \geq m$, we have that $c^q(x_0,x_1) = 0$ and $\ell^q(x_0) = \seq{1,u}$. But this makes $q$ an extension of $p$ in $D$, a contradiction. Hence, this case cannot obtain.
\end{proof}

It is easy to show that there is a $1$-generic for $\mathbb{P}$ computable in $\emptyset'$ and that every such generic is low. The proof is the same as for Cohen forcing; see~\cite[Lemma 2.6]{Jockusch-1980}.

\begin{corollary}
	There is a low stable coloring $c : [\omega]^2 \to 2$ such that for every $k \geq 2$, every stable coloring $d : [\omega]^2 \to k$ has an infinite limit-homogeneous set $L$ that computes no infinite homogeneous set for $c$.
\end{corollary}

\begin{proof}
	Take $S$ in the theorem to be a set of low PA degree, and let $c$ be any low $1$-$S$-generic for $\mathbb{P}$.	
\end{proof}

\noindent Note that since $\SRT^2_2$ and $\D^2_2$ are computably equivalent, we cannot hope to replace ``low'' above with ``computable''. The corollary can thus be seen as saying that this is a sharp division with respect to the jump hierarchy. 

\begin{corollary}\label{C:SRT22nscredD2}
	$\SRT^2_2 \nscred \D^2_{<\infty}$.	
\end{corollary}

\section{$\COH$ and uniform reducibility}\label{S:COH_uniform}

In this section, we look at uniform reducibility and the principle $\COH$, and among other things, show that $\COH$ is not uniformly reducible to $\SRT^2_{<\infty}$. The main technical ingredient will be Lemma~\ref{L:COHnuSRT} below. To state it, we first need some definitions. Say $\sigma' \in \omega^{<\omega}$ is a \emph{$1$-extension} of $\sigma \in \omega^{<\omega}$ if $\sigma' \succeq \sigma$ and $|\sigma'| = |\sigma| + 1$.

\begin{definition}\
	\begin{enumerate}
		\item A \emph{tree enumeration} is a partial computable function $U$ from $\omega$ to the finite subsets of $Inc(\omega)$ such that $U(0) \simeq \{\emptyset\}$, and if $U(x+1) \downarrow$ then $U(x) \downarrow$, and every string in $U(x+1)$ is a $1$-extension of some string in $U(x)$.
		\item An \emph{uniform sequence} of tree enumerations is a computable (set of indices for a) sequence of tree enumerations $\seq{U_l : l \in \omega}$ such that if $U_l(x) \downarrow$ and $U_{l'}(x') \downarrow$ for some $l' < l$ and numbers $x,x'$, then there is a stage $s$ such that $U_l(x)[s] \uparrow$ and $U_{l'}(x)[s] \downarrow$.
	\end{enumerate}
\end{definition}

\noindent Thus, all the members of $U(x)$ are strings of length $x$, and we think of $U$ as enumerating a subtree of $\omega^{<\omega}$, which we also identify with $U$. If $U$ is total as a function, then $U$ as a tree is actually computable, computably bounded, and infinite. And in this case, the the non-extendible nodes of $U$ are precisely those $\sigma$ that belong to some $U(x)$ but have no extension in $U(x+1)$. We say $\sigma$ \emph{looks extendible} at $s$ if it has an extension in the largest $x$ such that $U(x)[s] \downarrow$.

We can now think of a uniform sequence of tree enumerations as a sequence of attempts at enumerating an infinite tree, such that each attempt stops before the next attempt begins. In particular, at most one of these attempts can succeed. Observe also that there is a computable listing of all (indices of) uniform sequences of tree enumerations. We assume that if $U_l(x)[s] \downarrow$ for some $x$ then $l < s$, and say $U_l$ \emph{looks infinite} at $s$ if there is an $x$ such that $U_l(x)[s] \downarrow$ and $U_{l'}(0)[s] \uparrow$ for all $l' > l$. Note that if $U$ is infinite as a tree, then the range of every path through $U$ is infinite.

The way this notion will come up is as follows. Suppose $\varphi_0,\ldots,\varphi_{k-1}$ are $\Sigma^0_1$ formulas of finite sets. We can then search for a computable $\seq{\varphi_0,\ldots,\varphi_{k-1}}$-forest using the simple inductive method in the proof of Lemma~\ref{L:exforest}, which we call the \emph{canonical search}. Recall how this goes. For each $j$, we try to build a $\varphi_j$-sequence. As this sequence gets longer, we build more of its $\varphi_{j+1}$-generated subtree: each new member of the $\varphi_j$-sequence (i.e., each new $\varphi_j$-tree) allows us to build one more level of this subtree. If the subtree eventually becomes a $\varphi_{j+1}$-tree (i.e., every terminal node has a finite subset of its range satisfying $\varphi_{j+1}$), we can add it to a $\varphi_{j+1}$-sequence, and start the construction of a new $\varphi_{j+1}$-generated subtree. This gives rise to a uniform sequence of tree enumerations: we let $U_l$ enumerate the levels of the $l$th $\varphi_{j+1}$-generated subtree, until, if ever, it becomes a $\varphi_{j+1}$-tree. (For $\varphi_0$, we look at longer and longer segments of $\omega$, or whatever infinite set we are working inside, to try to build the $l$th $\varphi_0$-tree, so we let $U_l$ enumerate these initial segments.)

We can now state the lemma. Define the following fast-growing computable function, $\# : \omega \to \omega$. Let $\#(0) = 1$, and for $k \geq 1$ let
\[
\#(k) = (k+1) \cdot \#(k-1)^k.
\]
Thus, $\#(1) = 2$, $\#(2) = 3 \cdot 2^2 = 12$, $\#(3) = 4 \cdot 12^3 = 6912$, and so on.

\begin{lemma}\label{L:COHnuSRT}
	Fix $k \geq 1$ and finite sets $C_0 \subseteq C$ with $|C - C_0| = k$. Let $d$ be a partial computable coloring $[\omega]^2 \to C$, and let $\seq{U_l : l \in \omega}$ be an ordered sequence of tree enumerations. There exists a computable coloring $c : \omega \to k+1$, and for each $j < k$, a computable coloring $c_j : \omega \to \#(k-1)$, with the following properties.
	\begin{enumerate}
		\item If $d$ is total and stable, some $U_l$ is infinite, and $\lim_y d(x,y) \notin C_0$ for almost all $x$ in the range of each path through $U_l$, then $d$ has an infinite homogeneous set $H$ contained in the range of a path through $U_l$ such that $\Psi^H$ is not an infinite almost-homogeneous set for $c$ or one of the $c_j$.
		\item Indices for $c$ and the $c_j$ as computable functions can be found uniformly from $k$, canonical indices for $C_0$ and $C$, an index for $d$ as a partial computable function, an index for $\seq{U_l : l \in \omega}$, and an index for $\Psi$.
	\end{enumerate}
\end{lemma}

\noindent We proceed by induction, starting with $k = 1$. 

\subsection{The $k = 1$ case.}

Note that since $\#(k-1) = \#(0) = 1$ in this case, and since every set is homogeneous for the trivial coloring $\omega \to 1$, we build only the coloring $c : \omega \to k+1 = 2$ here.

\begin{proof}[Proof of Lemma~\ref{L:COHnuSRT} for $k = 1$]
	Assume for simplicity that $C - C_0 = \{0\}$. We build $c$, and for each $\sigma \in \omega^{<\omega}$, try to define a finite set $H_\sigma \subseteq \ran(\sigma)$ homogeneous for $d$ with color $0$ such that $H_{\sigma} > H_\sigma'$ for all $\sigma \succ \sigma'$. At each stage, all but finitely many of the $H_\sigma$ will be undefined, and each of those that are defined may subsequently become undefined and redefined any number of times. However, if $d$ satisfies the hypotheses of the lemma and $U_l$ is infinite, we ensure there is a path $P$ through $U_l$ satisfying one of the following two outcomes.
	\begin{itemize}
		\item $d$ has an infinite homogeneous set $H \subseteq \ran(P)$ such that $\Psi^H$ is not infinite.
		\item There exists $\sigma_0 \prec \sigma_1 \prec \cdots \prec P$ such that each of the sets $H_{\sigma_n}$ stabilizes to a finite set $H_{P_n}$, and $H = \bigcup_n H_{P_n} \subseteq \ran(P)$ is infinite and $\Psi^H$ is not almost-homogeneous for $c$.
	\end{itemize}
	Certainly this suffices to prove the lemma.
	
	We write $d(x,y)[s] \downarrow$ to mean that the coloring $d$ converges on $(x,y)$ in $s$ or fewer steps, and follow the convention that if $d(x,y)[s] \downarrow$ then also $d(x,y')[s] \downarrow$ for all $y'$ with $x < y' \leq y$. Given $x$, we computably guess at $\lim_y d(x,y)$ as follows: at stage $s$, choose the largest $y$ with $x < y \leq s$ such that $d(x,y)[s] \downarrow$, and guess $d(x,y)$ to be the limit. Thus, if $d$ is not actually a stable coloring, our guess about the limit may change infinitely often, but otherwise it will eventually be correct.
	
	
	\medskip
	\emph{Construction.} Initially, let all the $H_\sigma$ be undefined. At the start of stage $s$, assume we have defined $c$ on $\omega \res s$ and let $l$ be such that $U_l$ looks infinite. We only define $H_\sigma$ at this stage if $\sigma$ is a terminal node in $U_l$ that looks extendible at this stage. Alongside this definition, we also define a number $u_\sigma \geq H_\sigma$. By induction, we assume that each such $\sigma$ has the same number, say $n$, of (not necessarily proper) initial segments $\sigma'$ for which $H_{\sigma'}$ is defined already, and that if $H_{\sigma'}$ is defined then $\sigma'$ is such an initial segment. Given any $\sigma$ on $U_l$ (not necessarily terminal), we denote its $(m+1)$st initial segment $\sigma'$ such that $H_{\sigma'}$ is currently defined by $\sigma_m$, so that necessarily $n < m$ by assumption. Thus, if $\sigma' \prec \sigma$ and $\sigma'_m$ is defined then so is $\sigma_m$ and $\sigma'_m = \sigma_m$. We will also maintain that if $\sigma$ and $\sigma'$ are any two strings with $\sigma_m$ and $\sigma'_m$ defined then $|\sigma_m| = |\sigma'_m|$.
	
	Now for all stages $t \geq s$, we start coloring $c(t) = n \mod 2$, not defining or undefining any sets, until one or more of the following conditions applies at $t$.
	\begin{enumerate}
		\item Each terminal $\sigma$ in $U_l$ that looks extendible has $H_\sigma$ undefined (so $\sigma \neq \sigma_{n-1}$), and there is a finite set $F$ such that:
		\begin{itemize}
			\item $\emptyset \neq F \subseteq \ran(\sigma)$;
			\item $\max_{m < n} u_{\sigma_m} < F$;
			\item there is an $x \geq \max_{m < n} u_{\sigma_m}$ such that $\Psi^{\bigcup_{m < n} H_{\sigma_m} \cup F}(x) \downarrow = 1$;
			\item $d(x,y) \downarrow = 0$ for all $x < y$ with $x \in \bigcup_{m < n} H_{\sigma_m} \cup F$ and $y \in F$;
			\item all the $x \in F$ look like they limit to $0$ under $d$.
		\end{itemize}
		\item Some $\sigma_m$ no longer looks extendible.
		\item Some $x$ in some $H_{\sigma_m}$ looks like it limits to a color other than $0$.
		\item $U_l$ no longer looks infinite.
	\end{enumerate}
	
	By usual conventions, the first condition means that $t$ is larger than $\max F$, the number $x$, and the use of the computation $\Psi^{\bigcup_{m < n} H_{\sigma_m} \cup F}(x) \downarrow = 1$. In particular, $c(x) = n \mod 2$. If this condition applies, fix the least such $F$ in each $\sigma$, define $H_\sigma = F$, and define $u_\sigma = t$. If the second condition applies, then for any such $\sigma_m$ we undefine the set $H_{\sigma_m}$ and number $u_{\sigma_m}$. If the third condition applies, we choose the least $m$ for which there is some such $H_{\sigma_m}$, and then undefine $H_{\sigma'_{m'}}$ and $u_{\sigma'_{m'}}$ for all $\sigma'$ and all $m' \geq m$. Note that each terminal $\sigma$ on $U_l$ that looks extendible still has the same number of initial segments $\sigma'$ for which $H_{\sigma'}$ is defined, and in the first case it is $n+1$ many. If the fourth condition applies, undefine all $H_\sigma$ and $u_\sigma$, and start over with $U_{l+1}$ instead of $U_l$. This completes the construction.
	
	\medskip
	\noindent \emph{Verification.} Clearly, $c$ is total and computable. So suppose $d$ is is total and stable, $U_l$ is infinite, and almost all $x$ in the range of each path through $U_l$ limit to $0$ under $d$. If there is a path $P$ through $U_l$ and an infinite homogeneous set $H \subseteq \ran(P)$ for $d$ such that $\Psi^H$ is not an infinite set, then we are done. So suppose not. Let $s_0$ be the least stage after which $U_l$ looks infinite and none of the finitely many $x$ that do not limit to $0$ look like they do. Note that if $H_\sigma$ is defined after stage $s_0$, it is consequently to a set all of whose members limit to $0$ under $d$.
	
	For every $n$, we claim there is a stage $s \geq 0$ such that if $P$ is a path through $U_l$ and $\sigma \prec P$ is a terminal node of $U_l$ at stage $s$ (which necessarily looks extendible), then each $H_{\sigma_m}$ with $m < n$ is defined and has stabilized to some finite set $H_{P_m}$. This means, in particular, that every element of $H_{P_m}$ actually limits to $0$ under $d$. We prove the claim by induction. Fix $n$, and let $s$ witness that the claim holds for all $m < n$. We may assume that $s$ enumerates enough of $U_l$ so that no terminal $\sigma$ on $U_l$ equals $\sigma_{n-1}$, so the latter is a proper initial segment of $\sigma$. For each path $P$ through $U_l$ and $\sigma \prec P$, write $u_{P_m}$ for $u_{\sigma_m}$ if $H_{\sigma_m}$ has stabilized to $H_{P_m}$.
	If from some stage after $s$ onwards, some (and hence every) terminal $\sigma$ on $U_l$ that looks extendible has $H_\sigma$ undefined, then it must be that the first condition of the construction does not apply. By choice of $s_0$ and $s$, and the fact that $U_l$ is a finitely-branching tree, this can only happen if there is a number $z \geq \max_{m < n} u_{P_m}$ and a path $P$ through $U_l$ whose range has no finite subset $F > z$ homogeneous for $d$ with color $0$ such that $\Psi^{\bigcup_{m < n} H_{Q_m} \cup F}(x) \downarrow = 1$ for some $x \geq \max_{m < n} u_{P_m}$. But then let $H$ be any infinite homogeneous set for $d$ containing $\bigcup_{m < n} H_{P_m}$ and otherwise only elements of $\ran(P)$ bigger than $z$, and observe that $\Psi^H$ is bounded by $\max_{m < n} u_{P_m}$ and is thus finite, which is a contradiction. We conclude that there are infinitely many stages after $s$ at which the $H_\sigma$ for $\sigma$ terminal in $U_l$ and looking extendible are defined. And since, for any path $P$ through $U_l$ and any $\sigma \prec P$, such an $H_\sigma$ can later only be undefined if some element of it starts looking like it does not limit to $0$ (since it will always look extendible), it is now easy to see that one of these definitions must become permanent.
	
	To conclude the proof, fix any path $P$ through $U_l$. Note that since each $H_{P_n}$ is non-empty, $H = \bigcup_n H_{P_n}$ is infinite. Furthermore, for each $n$ there is an $x \geq \max_{m < n} u_{P_m} \geq \bigcup_{m < n} H_{P_m}$ such that $c(x) = n \mod 2$ and $\Psi^{\bigcup_{m \leq n} H_{P_m}}(x) \downarrow = 1$. Since $\min F_{m+1}$ is larger than the use of this computation, this also means that $\Psi^H(x) \downarrow = 1$. Thus, $\Psi^H$ contains infinitely many numbers colored $0$ by $c$, and infinitely many colored $1$, and so is not almost-homogeneous for $c$.
\end{proof}

\subsection{The $k > 1$ case.}

In the $k=1$ case of Lemma~\ref{L:COHnuSRT}, the coloring $d$ is essentially a $1$-coloring, so finding a homogeneous set for it is straightforward. In the general case, we must now instead use the technology of Section~\ref{S:gencore} to build homogeneous sets, which complicates the argument considerably.

\begin{proof}[Proof of Lemma~\ref{L:COHnuSRT} for $k > 1$]
	We focus on the differences with the $k = 1$ case. Assume the lemma for $k$, and for simplicity, assume also that $C - C_0 = \{0,\ldots,k-1\}$. Along with $c$, we try to define $k$ many finite sets $H_{\sigma,0},\ldots,H_{\sigma,k-1}$ for each $\sigma \in \omega^{< \omega}$, with $H_{\sigma,j}$ homogeneous for $d$ with color $j$. To build the $c_j$, we appeal to the inductive hypothesis. If $d$ and $U_l$ satisfy the hypotheses of the lemma, we ensure there is a path $P$ through $U_l$ satisfying one of the following outcomes.
	\begin{itemize}
		\item $d$ has an infinite homogeneous set $H \subseteq \ran(P)$ such that $\Psi^H$ is not infinite.
		\item For some $j < k$, $d$ has an infinite homogeneous set $H \subseteq \ran(P)$ such that $\Psi^H$ is not almost-homogeneous for $c_j$.
		\item There exists $\sigma_0 \prec \sigma_1 \prec \cdots \prec P$ such that for each $j < k$, each of the sets $H_{\sigma_n,j}$ stabilizes to a finite set $H_{P_n,j}$, and for some such $j$, $H = \bigcup_n H_{P_n,j} \subseteq \ran(P)$ is infinite and $\Psi^H$ is not almost-homogeneous for $c$.
	\end{itemize}
	
	\medskip
	\emph{Construction of $c$.} Whenever we define one of $H_{\sigma,0},\ldots,H_{\sigma,k-1}$ for some $\sigma$ then we define all of them, and exactly one of these sets is non-empty. As before, we also define a number $u_\sigma$, and we follow the same conventions and notations. Initially, let all the $H_{\sigma,j}$ be undefined. At the start of stage $s$, we assume that $c$ is defined on $\omega \res s$, and that every terminal $\sigma$ on $U_l$ that looks extendible has the same number, $n$, of initial segments $\sigma'$ for which $H_{\sigma',0},\ldots,H_{\sigma',k-1}$ are defined. For each such $\sigma$ and each $j < k$, let $\varphi_{\sigma,j}(F)$ be the $\Sigma^0_1$ formula asserting:
	\begin{itemize}
		\item $\emptyset \neq F \subseteq \ran(\sigma)$;
		\item $\max_{m < n} u_{\sigma_m} < F$;
		\item there is an $x \geq \max_{m < n} u_{\sigma_m}$ such that $\Psi^{\bigcup_{m < n}H_{\sigma_m,j} \cup F}(x) \downarrow = 1$;
		\item $d(x,y) \downarrow = j$ for all $x < y$ with $x \in \bigcup_{m < n} H_{\sigma_m,j} \cup F$ and $y \in F$.
	\end{itemize}
	Note that if $\sigma$ is terminal in $U_l$ at $s$ and $\sigma'$ is terminal at some $t \geq s$, then $\varphi_{\sigma,j}$ and $\varphi_{\sigma',j}$ are the same formula so long as we did not define or undefine any sets between stages $s$ and $t$. 
	
	For all stages $t \geq s$, we start coloring $c(t) = n \mod (k+1)$ until one of the following conditions applies at $t$.
	\begin{enumerate}
		\item Each terminal $\sigma$ in $U_l$ that looks extendible has $H_\sigma$ undefined and the canonical search has found a $\seq{\varphi_{\sigma,0},\ldots,\varphi_{\sigma,k-1}}$-forest.
		\item Some $\sigma_m$ no longer looks extendible.
		\item Some $x$ in some $H_{\sigma_m,j}$ looks like it limits to a color other than $j$.
		\item $U_l$ no longer looks infinite.
	\end{enumerate}
	If the first condition applies, consider any terminal $\sigma$ on $U_l$ that looks extendible, and let $\seq{T_{j,0} < \cdots < T_{j,s_j} : j < k }$ be the $\seq{\varphi_{\sigma,0},\ldots,\varphi_{\sigma,k-1}}$-forest inside it. Thus, for each $j$, every terminal $\alpha$ in every $T_{j,i}$ has a finite subset $F$ of its range satisfying $\varphi_{\sigma,j}$. Now if for each $\sigma$ there is a $j$ and an $F$ as above all of whose elements look like they limit to $j$ under $d$, then define $H_{\sigma,j} = F$ for the least such $j$ and $F$, and define $H_{\sigma,j'} = \emptyset$ for all $j' \neq j$ and $u_{\sigma} = t$. Otherwise, do nothing. If the second condition applies, then for any such $\sigma_m$ we undefine the sets $H_{\sigma_m,0},\ldots,H_{\sigma_m,k-1}$ and number $u_{\sigma_m}$. If the third condition applies, fix the least $m$ for which there is some such $H_{\sigma_m,j}$, and then undefine $H_{\sigma'_{m'},0},\ldots,H_{\sigma'_{m'},k-1}$ and $u_{\sigma'_{m'}}$ for all $\sigma'$ and all $m' \geq m$.
	Finally, if the fourth condition applies, undefine all $H_{\sigma,j}$. Naturally, if we undefine some $H_\sigma$ we terminate our canonical search in any extensions of $\sigma$. This completes the construction of $c$.
	
	\medskip
	\emph{Construction of $c_j$.} We first define an ordered sequence of tree enumerations, $\seq{\widehat{U}_m : m \in \omega}$. At the beginning of $s$, fix the largest $m$ such that we have defined $\widehat{U}_m(0)$ (necessarily to be $\{\emptyset\}$), and let $x_0 > 0$ be least with $\widehat{U}_m(x_0)$ still undefined. Let $l$ be such that $U_l$ looks infinite at stage $s$.
	In the construction of $c$ at this stage, we either begin, or are in the midst of, the canonical search for a $\seq{\varphi_{\sigma,0},\ldots,\varphi_{\sigma,k-1}}$-forest for each terminal $\sigma$ on $U_l$ that looks extendible. We assume the search has not found a new $\varphi_j$-tree inside the range of each such $\sigma$ since we started defining $\widehat{U}_m$. We then let $\widehat{U}_m(x_0)$ be undefined until one of the following conditions applies at some stage $t \geq s$. We divide the first condition in two depending as $j = 0$ or $j > 0$.
	\begin{enumerate}
		\item ($j=0$) $U_l$ has enumerated new terminal nodes that look extendible, and for at least one such node $\sigma$ the canonical search has found no new $\varphi_{\sigma,j}$-tree.
		\item[(1)] ($j>0$) $U_l$ has enumerated new terminal nodes that look extendible, for each such node $\sigma$ the canonical search has found a new $\varphi_{\sigma,j-1}$-tree, and for at least one such node $\sigma$ the canonical search has found no new $\varphi_{\sigma,j}$-tree.
		\item The canonical search has found a new $\varphi_j$-tree in the range of each terminal node in $U_l$ that looks extendible.
		\item The canonical search terminates (successfully or unsuccessfully) for each terminal node in $U_l$ that looks extendible.
	\end{enumerate}
	
	If the first condition applies with $j = 0$, fix any such $\sigma$ for which the canonical search has found no new $\varphi_{\sigma,j}$-tree. If $x_0 = 1$, let $\widehat{U}_m(x_0)$ enumerate $\sigma(|\sigma|-1)$ as a string of length $1$. Otherwise, assume inductively that $\widehat{U}_m(x_0 -1)$ enumerated the string $\sigma(z)\cdots\sigma(|\sigma|-2)$ for some number $z$, and let $\widehat{U}_m(x_0)$ enumerate $\sigma(z)\cdots\sigma(|\sigma|-2)\sigma(|\sigma|-1)$. By induction, if $\widehat{U}_m$ turns out to be infinite, then every path through $\widehat{U}_m$ will be a co-initial segment of some path through $U_l$. In particular, the range of any path through $\widehat{U}_m$ has range contained in the range of some path through $U_l$
	
	If the first condition applies with $j > 0$, consider any $\sigma$ for which the canonical search has found no new $\varphi_{\sigma,j}$-tree, and let $T$ be the least new $\varphi_{\sigma,j-1}$-tree that has been found. If $x_0 = 1$, let $\widehat{U}_m$ enumerate every element of $\ran(T)$. If $x_0 > 1$, assume inductively that there is a $\sigma' \prec \sigma$ such that $\widehat{U}_m(x_0 - 1)$ enumerated some string $\tau$ with $\tau(|\tau|-1)$ in a $\varphi_{\sigma',j-1}$-tree $T' < T$. For each such $\tau$ and each $x \in \ran(T)$, let $\widehat{U}_m(x_0)$ enumerate the string $\tau x$. In other words, for the longest initial segment of $\sigma$ that $\widehat{U}_m$ already enumerated some string $\tau$ for, it now enumerates all $1$-extensions of $\tau$ with last bit from $\ran(T)$. By induction, it follows that each string $\tau$ that $\widehat{U}_m$ enumerates on behalf of some $\sigma$ has $\ran(\tau)$ contained in a $\varphi_{\sigma,j-1}$-tree and hence in $\ran(\sigma)$. Thus, if $\widehat{U}_m$ is infinite, every path through it has range contained in the range of some path through $U_l$.
	
	If the second or third condition applies, leave $\widehat{U}_m(x_0)$ undefined and define instead $\widehat{U}_{m+1}(0) = \{\emptyset\}$. So in particular, we never define $\widehat{U}_m$ on $x_0$ or any larger numbers.
	
	It is easy to see that the $\widehat{U}_m$ indeed form an ordered sequence of tree enumerations. Next, let $C'_0 = C_0 \cup \{j\}$, so that $|C - C_0'| = k-1$. Apply the inductive hypothesis to the coloring $d$, the sets $C$ and $C'_0$, and $\seq{\widehat{U}_m : m \in \omega}$ to obtain a coloring $c' : \omega \to k$, and for each $j' < k-1$, a coloring $c'_{j'} : \omega \to \#(k-2)$. Finally, let $c_j : \omega \to \#(k-1)$ be defined by
	\[
		c_j(x) = \seq{c'(x), c'_0(x), \ldots, c'_{k-2}(x)}.
	\]
	Note that every infinite almost-homogeneous set for $c_j$ is also almost-homogeneous for $c'$ and each of the $c'_{j'}$. This completes the construction.
	
	\medskip
	\emph{Verification.} Suppose $d$ is total and stable, $U_l$ is infinite, and almost all $x$ in the range of each path through $U_l$ limit to a color in $C - C_0$ under $d$. As in the $k = 1$ case, we may assume there is no path $P$ through $U_l$ whose range contains an infinite homogeneous set $H$ for $d$ such that $\Psi^H$ is not infinite.
	
	Now fix any $j < k$, and suppose one of the tree enumerations $\widehat{U}_m$ defined in the construction of $c_j$ is infinite as a tree. Then at each stage $s$ after we start defining $\widehat{U}_m$ there must be a terminal $\sigma$ that looks extendible in $U_l$ such that the canonical search for a $\seq{\varphi_{\sigma,0},\ldots,\varphi_{\sigma,k-1}}$-forest in the construction of $c$ has not yet terminated. Hence, there is a path $P$ through $U_l$ such that this is true of every $\sigma \prec P$ that is terminal in $U_l$ (and necessarily looks extendible) at such a stage $s$. As noted in the construction, the formulas $\varphi_{\sigma,0},\ldots,\varphi_{\sigma,k-1}$ do not change between such initial segments $\sigma$ of $P$ while the canonical search is ongoing, so we can write simply $\varphi_{P,j}$ in place of $\varphi_{\sigma,j}$. Also, the number of $\sigma' \preceq \sigma$ for which $H_{\sigma'}$ is defined cannot change for any such $\sigma$, since doing so terminates the canonical search along $P$. So if this number is $n$, then for any sufficiently long initial segment $\sigma$ of $P$ we have that $H_{\sigma_m,j}$ is defined for each $m < n$, and we can write simply $H_{P_m}$ in place of $H_{\sigma_m}$, and $u_{P_m}$ in place of $u_{\sigma_m}$. In particular, every element of $H_{P_m,j}$ actually limits to $j$ under $d$ (possibly trivially so, if the set is empty), because otherwise this set would be eventually undefined and the search terminated.
	
	Fix any path $Q$ through $\widehat{U}_m$, and recall the construction of $c_j$. If $j = 0$, then $Q$ is a co-initial segment of some path $P$ through $U_l$, inside whose range we find no $\varphi_{P,j}$-trees.
	If $j > 0$, then for some path $P$ through $U_l$, we have that $Q$ is a path through an infinite $\varphi_{P,j}$-generated subtree of an infinite $\varphi_{P,j-1}$-sequence.
	Either way, this means that no finite $F \subseteq \ran(Q)$ satisfies $\varphi_{P,j}$. Either way, this means that if $F \subseteq \ran(Q)$ is homogeneous for $d$ with color $j$ then there is no $x \geq \max_{m < n} u_{P_m}$ with $\Psi^{\bigcup_{m < n} H_{P_m,0} \cup F}(x) \downarrow = 1$. Now if infinitely many elements of $\ran(Q)$ limit to $j$ under $d$, then let $H$ be any infinite homogeneous set for $d$ with color $j$ that contains $\bigcup_{m < n} H_{P_m,j}$ and otherwise only contains elements of $\ran(Q)$. Then $H$ is a subset of $\ran(P)$ and $\Psi^H$ is bounded by $\max_{m < n} H_{P_m}$ and so is finite, contradicting our assumption above.
	
	We conclude that if $Q$ is any path through $\widehat{U}_m$, then almost all $x$ in $\ran(Q)$ limit to a color other than $j$ under $d$. Since $\ran(Q)$ is a subset of $\ran(P)$ for some path $P$ through $U_l$, this means that almost all $x$ in $\ran(Q)$ limit to one of $1,\ldots,k-1$, or equivalently, to a member of $C - C'_0$, as defined in the construction of $c_j$. But then we are precisely in the hypothesis of the lemma for $k-1$, and $d$ consequently has an infinite homogeneous set $H$, contained in the range of some infinite path through $\widehat{U}_m$ and hence some infinite path through $U_l$, such that $\Psi^H$ is not almost-homogeneous for $c_j$.
	
	Going forward, we may thus assume that no $\widehat{U}_m$ defined in the construction of any $c_j$ is infinite. Let $s_0$ be the least stage after which $U_l$ looks infinite and none of the $x$ that do not limit to a color in $C - C_0$ under $d$ look like they do. For every $n$, we claim there is a stage $s \geq s_0$ such that if $P$ is a path through $U_l$ and $\sigma \prec P$ is a terminal node of $U_l$ at stage $s$, then each $H_{\sigma_m,j}$ with $m < n$ has stabilized to some finite set $H_{P_m,j}$. Fix $n$ and assume the claim for all $m < n$, as witnessed by some $s \geq s_0$. We may assume $s$ is large enough so that no terminal $\sigma$ on $U_l$ that looks extendible is equal to $\sigma_{n-1}$.
	
	Suppose first that at all sufficiently large stages $t \geq s$, some (and hence every) terminal $\sigma$ on $U_l$ that looks extendible has $H_{\sigma,0},\ldots,H_{\sigma,k-1}$ undefined. By choice of $s_0$ and $l$, conditions 3 and 4 of the construction cannot apply at any such $t$, and whenever condition 2 applies it does not change how many initial segments $\sigma'$ with $H_{\sigma'}$ defined the terminal nodes that still look extendible have. Thus, as above, the formulas $\varphi_{\sigma,0},\ldots,\varphi_{\sigma,k-1}$ do not change between compatible strings $\sigma$ at these stages $t$. This means that if the canonical search ever finds a $\seq{\varphi_{\sigma,0},\ldots,\varphi_{\sigma,k-1}}$-forest at any of these stages but we do not define $H_{\sigma,0},\ldots,H_{\sigma,k-1}$, then the same forest is found again at the next stage, and also for any newly enumerated extensions of $\sigma$. But the only reason we might fail to define new sets when a forest is found is if for some $\sigma$ that looks extendible there is no finite $F$ in the range of any terminal $\alpha$ in any $\varphi_{\sigma,j}$-tree such that $F$ satisfies $\varphi_{\sigma,j}$ and every element of $F$ looks it limits to $j$ under $d$. By Lemma~\ref{L:combcore} and the fact that $U_l$ is finitely-branching, this is impossible.
	
 	It follows that at any of these stages $t$, there is a terminal $\sigma$ in $U_l$ that looks extendible such that the canonical search has not found a $\seq{\varphi_{\sigma,0},\ldots,\varphi_{\sigma,k-1}}$-forest. (That is, $\sigma$ witnesses that condition 1 of the construction does not apply at $t$.) There are two cases that can cause this situation.
 	
 	Case 1: at any stage $t$ as above, there is a terminal $\sigma$ in $U_l$ that looks extendible such that the canonical search does not find a new $\varphi_{\sigma,0}$-tree. Let $m$ be such that at the least stage $t$ as above, we are defining $\widehat{U}_m$ in the construction of $c_0$. Then it is easy to see that $\widehat{U}_m$ is infinite.
 	
 	Case 2: for some $j > 0$, there are infinitely many stages $t$ as above such that the canonical search finds a new $\varphi_{\sigma,j-1}$-tree for each terminal $\sigma$ in $U_l$ that looks extendible, but at each such $t$ there is at least one $\sigma$ such that this search does not find a new $\varphi_{\sigma,j}$-tree. Let $m$ be such that at the least stage $t$ as above, we are defining $\widehat{U}_m$ in the construction of $c_j$. Then $\widehat{U}_m$ is infinite.
 	
 	Since both cases result in contradictions, we conclude that there are infinitely many stages after $s$ at which the $H_\sigma$ for $\sigma$ terminal in $U_l$ and looking extendible are defined. For any path $P$ through $U_l$ and any $\sigma \prec P$, such an $H_{\sigma,j}$ can later only be undefined if some element of it starts looking like it does not limit to $j$ under $d$. Since the $H_{\sigma,j}$ are always chosen from a certain $\seq{\varphi_{\sigma,0},\ldots,\varphi_{\sigma,k-1}}$-forest, it follows by Lemma~\ref{L:combcore} and the fact that $d$ is stable that this can only happen finitely often. Hence, this definition eventually stabilizes.
	
	To complete the proof, fix any path $P$ through $U_l$. Let $\sigma_0 \prec \sigma_1 \prec \cdots \prec P$ be the strings for which we just showed that the sets $H_{\sigma_n,j}$ stabilize to $H_{P_n,j}$. For each $n$, there is a unique $j < k$ such that $H_{P_n,j} \neq \emptyset$. Hence, for each $i < k + 1$, there is a $j < k$ such that $H_{P_n,j} \neq \emptyset$ for infinitely many $n \equiv i \mod (k+1)$; let $j_i$ be the least such $j$. Then we can fix $i < i' < k+1$ and $j$ with $j_i = j_{i'} = j$. Since $H_{P_n,j} \neq \emptyset$ for infinitely many $n$, and since $H_{P_m,j} < H_{P_n,j}$ whenever $m < n$ and both sets are non-empty, it follows that $H = \bigcup_n H_{P_n,j}$ is infinite. In particular, $H$ is an infinite homogeneous set for $d$ with color $j$. Furthermore, whenever some $H_{P_n,j}$ is non-empty there is an $x \geq \max_{m < n} u_{P_m}$ such that $\Psi^{\bigcup_{m < n} H_{P_m,0} \cup F}(x) \downarrow = 1$ and $c(x) = n \mod (k+1)$. Note that $\Psi^H$ agrees with this computation by construction. As there are infinitely many $n$ with $H_{P_n,j} \neq \emptyset$ and $n \equiv i \mod (k+1)$, and infinitely many $n$ with $H_{P_n} \neq \emptyset$ and $n \equiv i' \mod (k+1)$, we conclude that there are infinitely many $x$ in $\Psi^H$ with $c(x) = i$, and infinitely many $x$ with $c(x) = i'$. Hence, $\Psi^H$ is not almost-homogeneous for $c$, as was to be shown. This completes the proof.
\end{proof}

\subsection{Consequences}

We can now give the the main result of this section, which is just a simplified version of the lemma just proved.

\begin{theorem}\label{T:nured_simple}
	Fix $k \geq 1$. Let $d$ be a partial computable coloring $[\omega]^2 \to k$ and let $\Psi$ be a Turing functional. There exists a computable coloring $e : \omega \to \#(k)$ such that if $d$ is total and stable, then it has an infinite homogeneous set $H$ for which $\Psi^H$ is not almost-homogeneous for $e$.	 Moreover, an index for $e$ can be obtained from $k$, an index for $d$ as a partial computable function, and an index for $\Psi$.
\end{theorem}

\begin{proof}
	Let $\seq{U_l : l \in \omega}$ be any uniform sequence of tree enumerations with $U_0(x) \downarrow = \{\omega \res x\}$ for all $x$, and let $C = \{0,\ldots,k\}$ and $C_0 = \emptyset$. Then, apply Lemma~\ref{L:COHnuSRT} to $k$, $C_0$ and $C$, $d$, and $\seq{U_l : l \in \omega}$ to get colorings $c : \omega \to k+1$ and, for each $j < k$, $c_j : \omega \to \#(k-1)$. Define $e : \omega \to \#(k)$ by
	\[
		e(x) = \seq{c(x),c_0(x),\ldots,c_{k-1}(x)}.\qedhere
	\]
\end{proof}

\begin{corollary}
	There exists a computable family of sets $\vec{X}$ with the following property. For every $k \geq 1$, if $d$ is a computable stable coloring $[\omega]^2 \to k$ and $\Psi$ is any Turing functional, then $d$ has an infinite homogeneous set $H$ such that $\Psi^H$ is not an infinite $\vec{X}$-cohesive set.
\end{corollary}

\begin{proof}
	By the uniformity of Theorem~\ref{T:nured_simple}, there is a computable sequence $\seq{e_i : i \in \omega}$ such that if $i$ is the triple (of indices for) $\seq{k,d,\Phi}$, then $e_i$ is the coloring $e : \omega \to \#(k)$ given by the theorem. We define a family of sets $\vec{X} = \seq{X_n : n \in \omega}$ as follows. Fix $i = \seq{k,d,\Phi}$, and for each $s$, let $A_{\seq{i,s}}(x)$ be the $s$th digit in the binary expansion of $c_i(x)$, regarded as a sequence of length $\lfloor \log_2 \#(k) \rfloor + 1$ by prepending $0$s if necessary, or $0$ if $s \geq \lfloor \log_2 \#(k) \rfloor + 1$. (For example, if $k = 2$ then $\#(k) = 12$, so if $e_i(x) = 3$ then $A_{\seq{i,0}}(x),A_{\seq{i,1}}(x),\ldots$ equal $0,0,1,1,0,0,0,\ldots$, respectively; if $e_i(x) = 9$ then $A_{\seq{i,0}}(x),A_{\seq{i,1}}(x),\ldots$ equal $1,0,0,1,0,0,0,\ldots$; etc.) Then every infinite set which is cohesive for $\seq{A_{\seq{i,s}} : s \in \omega}$ is an infinite almost-homogeneous set for $e_i$. So by definition of $e_i$, if $d$ is total and stable then it must have an infinite homogeneous set $H$ with $\Psi^H$ not equal to any infinite $\vec{X}$-cohesive set.
\end{proof}

Observe above the dependence of $H$ on $\Psi$. If this dependence could be eliminated, which is to say, if in the statement of the corollary we could interchange the universal quantifier over $\Psi$ with the existential quantifier over $H$, we would have that $\COH \ncred \SRT^2_{<\infty}$. This remains an open question. We do however have the following consequence as a special case.

\begin{corollary}\label{C:COHnuredSRT2}
	$\COH \nured \SRT^2_{<\infty}$.	
\end{corollary}

We note also the following direct consequence of Theorem~\ref{T:nured_simple}, which is of independent interest.

\begin{corollary}
	For all $k \geq 1$, we have that $\RT^1_{\#(k)} \nured \SRT^2_k$.
\end{corollary}

\noindent We do not know if this can be improved to show that $\RT^1_k \nured \SRT^2_j$ whenever $j < k$. This is the uniform analogue of Question 5.4 of Hirschfeldt and Jockusch~\cite{HJ-TA}, whether $\RT^1_k \scred \SRT^2_j$. Of course, $\RT^1_k$ is strongly uniformly reducible even to $\D^2_k$, and hence to $\SRT^2_k$.

We mention, in closing this section, that we do not know how to extend Corollary~\ref{C:COHnuredSRT2} from uniform reducibility to generalized uniform reducibility (as defined in~\cite{HJ-TA}, Definition 4.3). Whether this can be done appears also as Question 5.2 of~\cite{HJ-TA}. 

\section{$\COH$ and strong computable reducibility}\label{S:COH_strong}

For our final result, we turn to strong computable reducibility. In trying to show that, say, $\mathsf{Q} \nscred \mathsf{P}$, one may hope to be able to keep the construction of a witnessing instance $X$ of $\mathsf{Q}$ separate from the construction of a solution $\widehat{Y}$ to the computed instance $\Phi^X$ of $\mathsf{P}$. This is because the ``backward'' reduction $\Psi^{\widehat{Y}}$ does not reference $X$, and so one can try to build a bit of $\widehat{Y}$, then build a bit more of $X$ to diagonalize $\Psi^{\widehat{Y}}$, and then repeat for the next reduction procedure. This is indeed what happens for example in the proof that $\COH \nscred \D^2_2$ in~\cite{Dzhafarov-2015}, which builds a limit-homogeneous set independently of a family of sets. Unfortunately, here our $\mathsf{P}$ is $\SRT^2_2$ rather than $\D^2_2$, so we must build a homogeneous set rather than merely a limit-homogeneous one. This necessarily brings the construction of the original instance $X$ into the construction of the solution $\widehat{Y}$. To be specific, we can try to define a forest of some sort as we did above, then extend the coloring to force the limiting color of each number in the range of this forest, and appeal to Lemma~\ref{L:combcore} to conclude that the range of some path in some tree in this forest is limit-homogeneous. But there is no guarantee that this range is also homogeneous, let alone homogeneous with the same color as the limiting color. If we instead try to do this in the opposite order, by first defining a finite set with certain desirable properties (like causing $\Psi$ to converge on a new element), we can later extend the coloring to make this finite set homogeneous, but then there may be no further extension that causes all elements of this set to have the same limit.

The above is a serious obstacle. To get around it, we abandon the generalized Seetapun framework from Section~\ref{S:gencore}, and instead present the following alternative argument, which uses an entirely new method for building homogeneous sets. The drawback is that, unlike in our results above, where we constructed instances that were computable or close to computable, here our instance is far more complicated. We begin with a definition.

\begin{definition}
Let $\mathbb{C}$ be the following notion of forcing. A condition is a sequence $p = \seq{\sigma^p_0,\ldots,\sigma^p_{n^p-1},\ell^p}$, in which $\sigma^p_0,\ldots,\sigma^p_{n^p-1}$ are finite binary sequences of the same length, and $\ell^p$ is a function $n^p \to (2 \times \omega) \cup \{u\}$ such that if $\ell^p(n) = \seq{i,k}$ then $\sigma^p_n(x) = i$ for all $x$ with $k \leq x < |\sigma^p_{n}|$. A condition $q$ extends $p$ if $n^q \geq n^p$, $\sigma^q_n \succeq \sigma^p_n$ for all $n < n^p$, and $\ell^q \supseteq \ell^p$.
\end{definition}

\noindent We call each $\sigma^p_n$ a \emph{column} of $p$. The role of $\ell^p$ is to either \emph{lock} a column to some $i < 2$ from some point $k$ onward, or to \emph{unlock} it, meaning that it can never later be locked. Note that any sufficiently generic filter $\mathscr{G}$ for $\mathbb{C}$ gives rise to a family $\vec{X} = \seq{X_n : n \in \omega}$ of sets, with the $n$th column of any condition in $\mathscr{G}$ being an initial segment of $X_n$. More generally, we say $\vec{X}$ \emph{extends} $p$ provided each of $A_0,\ldots,A_{n^p-1}$ extend $\sigma^p_0,\ldots,\sigma^p_{n^p-1}$, respecting all locks. This means that if $\vec{X}$ extends $p$, which has locked its $n$th column to $i$ from $k$ on, then $X_n(x) = i$ for all $x \geq k$.

In what follows, we use $\vec{X}$ (respectively, $X_n$ for some $n \in \omega$) both for a family of sets we are building (respectively, for its $n$th column) and as a name in the $\mathbb{C}$ forcing language for a generic family (respectively, a name for the $n$th column of a generic family). We define what it means for a condition to force an arithmetical statement relative to this name in the obvious way: thus, we say $p$ forces $x \in X_n(x)$ if $x < |\sigma^p_n|$ and $\sigma^p_n(x) = 1$, and from there we proceed inductively in the usual manner. (See e.g., Shore~\cite[Chapter 3]{Shore-TA} for details.) Given any set $P$, we can also extend our forcing language and relation to include $P$ as a parameter.

\begin{theorem}\label{T:nscred}
There exist a family of sets $\vec{X} = \seq{X_n : n \in \omega}$ and a collection $Y$ of infinite sets such that no $(\vec{X} \oplus P)$-computable infinite set is $\vec{X}$-cohesive for any $P \in Y$, and every stable coloring of pairs computable from $\vec{X}$ either has an $(\vec{X} \oplus P)$-computable infinite homogeneous set for some $P \in Y$, or else, for each $j < 2$, an infinite homogeneous set $H_j$ with color $j$ that computes no infinite $\vec{X}$-cohesive set.
\end{theorem}

\begin{proof}
	The idea is to make $\vec{X}$ generic over the collection $Y$, where $Y$ will be obtained by repeatedly taking paths through certain non-well-founded trees of $\omega^{<\omega}$. We define
	\begin{itemize}
	\item a sequence of $\mathbb{C}$-conditions $p_0 \geq p_1 \geq \cdots$ with $\lim_s n^{p_s} = \infty$;
	\item a sequence of finite sets $H^\Phi_{j,0} \subseteq H^\Phi_{j,1} \subseteq \cdots$ for each Turing functional $\Phi$ and each $j < 2$;
	\item a sequence of infinite sets $I_0 \supseteq I_1 \supseteq \cdots$ with $H_{j,s} < I_s$ for each $j$ and all $s$;
	\item a sequence of finite families $Y_0 \subseteq Y_1 \subseteq \cdots$ of infinite subsets of $\omega$. 
	\end{itemize}
	In the end, we take $X_n = \bigcup_s \sigma^{p_s}_n$ for each $n$, and let $\vec{X} = \{X_n : n \in \omega\}$, let $H^\Phi_j = \bigcup_s H^\Phi_{j,s}$ for each $j < 2$, and let $Y = \bigcup_s Y_s$. Our goal is to ensure the following requirements, for all $s \in \omega$, all Turing functionals $\Phi$ and $\Psi$, and each $j < 2$:
	\begin{center}
		\begin{tabu}{llX}
			$\mathcal{P}_{s}$ & : & the sequence $p_0 \geq p_1 \geq \cdots$ is $3$-generic relative to each $P \in Y_s$;\\
			$\mathcal{Q}_{\Phi,s}$ & : & if $\Phi^{\vec{X}}$ is a stable coloring of pairs, it either has an $(\vec{X} \oplus P)$-computable infinite homogeneous set for some $P \in Y$, or both $H^\Phi_0$ and $H^\Phi_1$ are infinite;\\
			$\mathcal{R}_{\Phi,\Psi,i}$ & : & if $\Phi^{\vec{X}}$ is a stable coloring of pairs, it either has an $(\vec{X} \oplus P)$-computable infinite homogeneous set for some $P \in Y$, or if $\Psi^{H^\Phi_j}$ defines an infinite set then this set has infinite intersection with both $A_0$ and $\overline{A_0}$.
		\end{tabu}
	\end{center}
	The $\mathcal{P}$ requirements will ensure that for all $P \in Y$ there is no $(\vec{X} \oplus P)$-computable infinite $\vec{X}$-cohesive set. To see this, fix any $P \in Y$ and any Turing functional $\Gamma$. Let $W$ be the set of conditions $p$ that force one of the following two statements:
	\begin{itemize}
		\item $\Gamma^{\vec{X} \oplus P}$ does not define an infinite set;
		\item there is an $n \in \omega$ such that for every $k \in \omega$ and every $ \RADOMIL < 2$, there is an $x \geq k$ with $X_n(x) =  \RADOMIL$ and $\Gamma^{\vec{X} \oplus P}(x) \downarrow = 1$.
	\end{itemize}
	Then $W$ is $\Sigma^0_3$-definable in $P$, and we claim that it is dense in $\mathbb{C}$. Hence, some condition $p_s$ in our sequence must meet it, from which it follows that $\Gamma^{\vec{X} \oplus P}$ is either not an infinite set or not $\vec{X}$-cohesive, as desired. So let $p_0$ be any condition, and assume it has no extension forcing the first statement above. Let $q$ be any extension of $p$ with $n^q = n^p + 1$ and $\ell^q(n^p) = u$; we claim that $q$ forces the second statement, witnessed by $n = n^p$. Assume not, so that for some $k \in \omega$, some $ \RADOMIL < 2$, and some condition $r$ extending $q$, no extension of $r$ forces that there is an $x$ with $X_n(x) =  \RADOMIL$ and $\Gamma^{\vec{X} \oplus P}(x) \downarrow = 1$. Then in particular, there is no sequence of strings $\tau_0,\ldots,\tau_{b-1}$ as follows:
	\begin{enumerate}
		\item $b \geq n^r + 1$;
		\item each $\tau_m$ has the same length;
		\item each $\tau_m$ with $m < n^r$ extends $\sigma^r_n$ and respects any lock on this column;
		\item $\tau_{n^p}(x) =  \RADOMIL$ for all $x$ with $\max \{|\sigma^r_0|, k\} \leq x < |\tau_{n^p}|$;
		\item $\Gamma^{\seq{\tau_0,\ldots,\tau_{b-1}} \oplus I_t^{(e)} \res |\tau_0|}(x) \downarrow = 1$ for some $x$ with $\max \{|\sigma^r_0|, k\} \leq x < |\tau_{n^p}|$.
	\end{enumerate}
	Let $r'$ be the condition that differs from $r$ only in that $\ell^{r'}(n^p)$ is $\seq{\max \{|\sigma^r_0|, k\},j}$ rather than $u$. Then sequences of strings satisfying properties (1)--(4) above are precisely the sequences of columns of extensions of $r'$. Thus, the fact that no such sequence can also satisfy (5) means $r'$ forces that $\Gamma^{\vec{X} \oplus P}(x) \simeq 0$ for all $x \geq 0$, which is to say that this computation does not define an infinite set. Since $r'$ is an extension of $p$, this is impossible.
	
	It follows that the $\mathcal{Q}$ and $\mathcal{R}$ requirements ensure that every $\vec{X}$-computable stable coloring of pairs has an infinite homogeneous set that computes no infinite $\vec{X}$-cohesive set, as desired.
	
	\medskip
	\emph{Construction.} Distribute all the requirements between the stages of the construction in such a way that there are infinitely many stages dedicated to satisfying each requirement. We begin by letting $p_0$ be any condition with $n^{p_0} = 1$ and $\ell^{p_0}(0) = u$. Thus, we will be free to extend the $0$th column of any of our conditions $p_s$ arbitrarily. Let $H^\Phi_{0,0} = H^\Phi_{1,0} = \emptyset$ for all $\Phi$, let $I_0 = \omega$, and let $Y_0 = \emptyset$. At the beginning of stage $s+1$, assume we are given $p_s$, $H^\Phi_{0,s}$ and $H^\Phi_{1,s}$ for all $\Phi$, $I_s$, and $Y_s$. Assume inductively that if $H^\Phi_{0,s}$ or $H^\Phi_{1,s}$ is non-empty for some $\Phi$, then $p_s$ forces that $\Phi^{\vec{X}}$ is a stable coloring of pairs and $\Phi^{\vec{X}}(x,y) = j$ for all $x \in H^\Phi_{j,s}$ and all $y \in I_s$. At the conclusion of the stage, if we did not explicitly define $p_{s+1}$, $H^\Phi_{j,s+1}$ for some $\Phi$ or $j$, $I_{s+1}$, or $Y_{s+1}$, we mean to let these be $p_s$, $H^\Phi_{j,s}$, $I_s$, and $Y_s$, respectively.
	
	\medskip
	\emph{$\mathcal{P}$ requirements.} These are satisfied in a straightforward manner. Suppose $s$ is dedicated to requirement $\mathcal{P}_t$ for some $t < s$, and that it is the $\seq{n,m}$th such stage. If $n > |Y_t|$, do nothing. Otherwise, let $P$ be the $n$th member of the finite set $Y_t$ in some fixed listing, and let $W$ be the $m$th $\Sigma^0_3(P)$ set. If $p_s$ has an extension in $W$, choose one and let it be $p_{s+1}$, and otherwise do nothing.
	
	\medskip
	\emph{$\mathcal{Q}$ requirements.} Suppose $s$ is dedicated to $\mathcal{Q}_\Phi$. By extending $p_s$ if necessary, we may assume that it decides whether or not $\Phi^{\vec{X}}$ is a stable coloring of pairs. If $p_s$ forces that it is not such a coloring, we can do nothing. So assume otherwise. If, for some $j < 2$ and some $k \in \omega$, there is no extension of $p_s$ forcing that $\lim_y \Phi^{\vec{X}}(x,y) = j$ for some $x \geq k$, then $P = \{x \in I_s : x \geq k\}$ will be limit-homogeneous for $\Phi^{\vec{X}}$ with color $1-j$. Hence, $\vec{X} \oplus P$ will compute an infinite homogeneous set for $\Phi^{\vec{X}}$ via the uniform thinning algorithm for obtaining a homogeneous set from a limit-homogeneous one. We thus define $Y_{s+1} = Y_s \cup \{P\}$, and otherwise do nothing. If there are no $j$ and $k$ as above, we can find numbers $x_0,x_1 \in I_s$ and an extension of $p_s$ forcing that $\lim_y \Phi^{\vec{X}}(x_i,y) = j$ for each $j < 2$; we then let $p_{s+1}$ be this extension, and let $H^\Phi_{j,s+1} = H^\Phi_{j,s} \cup \{x_i\}$. Thus, we have added an element to each of $H^\Phi_0$ and $H^\Phi_1$, so if $\Phi^{\vec{X}}$ does not end up having an $(\vec{X} \oplus P)$-computable infinite homogeneous set for any $P \in Y$, both $H^\Phi_0$ and $H^\Phi_1$ will be infinite.
	
	\medskip
	\emph{$\mathcal{R}$ requirements.} Suppose $s$ is dedicated to $\mathcal{R}_{\Phi,\Psi,i}$. As above, we may assume $p_s$ forces that $\Phi^{\vec{X}}$ is a stable coloring of pairs. Let $k$ be the length of the columns of $p_s$, and define $T$ to be the tree of all $\alpha \in Inc(I_s)$ such that there is no finite $F \subseteq \ran(\alpha \res |\alpha| - 1)$ and no $w \geq k$ with $\Psi^{H^\Phi_{j,s} \cup F}(w) \downarrow = 1$. If $T$ is not well-founded, we let $I_{s+1}$ be the range of any infinite path through it. Then provided we ensure that all $x \in H^\Phi_j - H^\Phi_{j,s}$ come from $I_s$, any set defined by $\Psi^{H^\Phi_j}$ will contain no numbers $w \geq k$, and so will not be infinite.
	
	We thus turn to dealing with the case that $T$ is well-founded. Let $ \RADOMIL < 2$ be such that the number of prior stages dedicated to requirement $\mathcal{R}_{\Phi,\Psi,i}$ is congruent to $ \RADOMIL$ modulo $2$. Our strategy is to either define $H^\Phi_{j,s+1}$ so that $\Psi^{H^\Phi_{j,s+1}}(w) \downarrow = 1$ for some $w \geq k$ and ensure that $A_0(x) =  \RADOMIL$, or else to find a set $P$ such that $\Phi^{\vec{X}}$ has an $(\vec{X} \oplus P)$-computable infinite homogeneous set, and add this $P$ to $Y$.
	
	Notice each $\alpha \in T$ is either terminal, in which case there is an $F \subseteq \ran(\alpha)$ and a $w \geq k$ with $\Psi^{H^\Phi_{j,s} \cup F}(w) \downarrow = 1$, or else every $1$-extension of $\alpha$ is also in $T$.
	
	We wish to define a certain subtree $T_0$ of $T$. To this end, we first label each node of $T$ by either a finite number or the symbol $\infty$. We do this by induction on rank. If $\alpha \in T$ is terminal, label it by the least $w \geq k$ such that $\Psi^{H^\Phi_{j,s} \cup F}(w) \downarrow = 1$ for some $F \subseteq \ran(\alpha)$. Now suppose $\alpha \in T$ is not terminal, and that we have labeled all nodes of $T$ of smaller rank. In particular, this means that we have labeled all the $1$-extensions of $\alpha$ in $T$. If there is a $w \in \omega$ such that infinitely many of the $1$-extensions of $\alpha$ were labeled by $w$, we label $\alpha$ by the least such $w$. Otherwise, we label $\alpha$ by $\infty$.
	
	Now define $T_0$ as follows. Add $\emptyset$ to $T_0$, and suppose we have added some $\alpha \in T$ to $T_0$, and that this $\alpha$ is not terminal in $T$. If $\alpha$ was labeled by a finite number $w$, add to $T_0$ all $1$-extensions of $\alpha$ that were also labeled by $w$. If, on the other hand, $\alpha$ is labeled by $\infty$, and there are infinitely many $1$-extensions of $\alpha$ labeled by finite numbers, add these extensions to $T_0$. In this case, for each $w$, cofinitely many of the $1$-extensions $\beta$ of $\alpha$ must satisfy $\Psi^{H^\Phi_{j,s} \cup F}(x) \simeq 0$ for all $x$ with $k \leq x < w$ and all $F \subseteq \ran(\beta)$. If cofinitely many of the $1$-extensions of $\alpha$ were labeled by $\infty$, add these to $T_0$ instead. Thus, in any case, each non-terminal node in $T_0$ has infinitely many $1$-extensions in $T_0$, and all of these extensions have the same kind of label: they are either all labeled by one and the same finite number, provided $\alpha$ itself is labeled by this number; they are labeled by finite numbers, and only finitely many are labeled by any given number; or they are all labeled by $\infty$. Note also that every terminal node in $T_0$ is terminal in $T$.
	
	We now attempt to define a sequence $p_s \geq q_0 \geq q_1 \geq \cdots$ of conditions, and a sequence $\emptyset = \alpha_0 \preceq \alpha_1 \cdots$ of $1$-extensions in $T_0$, such that for each $n$, the condition $q_n$ forces that $\lim_y \Phi^{\vec{X}}(x,y) = j$ for all $x \in \ran(\alpha_n)$, and in fact, that $\Phi^{\vec{X}}(x,y) = j$ for all $x \in \ran(\alpha_n)$ and all $y \geq \alpha_{n+1}(n)$. Hence, $q_n$ forces that $\ran(\alpha_n)$ is homogeneous for $\Phi^{\vec{X}}$ with color $j$. If we get stuck in this process, we will be able to add some $P$ to $Y$; otherwise, we will succeed in defining $H^\Phi_{j,s+1}$.
	
	If $\alpha_0 = \emptyset$ is labeled by some $w \in \omega$, let $q_0$ be any extension of $p_s$ with $\sigma^{q_0}_0(w) =  \RADOMIL$, which exists since $w \geq k = |\sigma^{p_s}_0|$. If $\alpha_0$ is labeled by $\infty$, let $q_0 = p_s$. Next, assume we have defined $q_n$ and $\alpha_n$ for some $n$, and that $\alpha_n$ is not terminal in $T_0$. By assumption, $q_n$ forces that there is an $m \in \omega$ such that $\Phi^{\vec{X}}(x,y) = j$ for all $x \in \ran(\alpha_n)$ and all $y \geq m$. Let $S$ be the set of all $1$-extensions $\beta$ of $\alpha_n$ with $\beta(n) \geq m$, so that $S$ is infinite. We consider the following two cases.
	
	Case 1: $\alpha_n$ is labeled by $\infty$, but its $1$-extensions are all labeled by finite numbers. Let $P$ be the set of all pairs $\seq{x,w}$ such that $x$ is equal to $\beta(n)$ for some $\beta \in S$ with label $w \geq |\sigma^{q_n}_0|$. As noted above, there are infinitely many numbers $w$ with $\seq{x,w} \in P$ for some $x$. Now if there is an extension $q$ of $q_n$ and an $\seq{x,w} \in P$ such that $\sigma^q_0(w) =  \RADOMIL$ and $q$ forces that $\lim_y \Phi^{\vec{X}}(x,y) = j$, let $q_{n+1} = q$, and let $\alpha_{n+1}$ be any $\beta \in S$ with $\beta(n) = x$ and label $w$. Otherwise, we have that if $q$ is any condition extending $q_n$ with $\sigma^q_0(w) =  \RADOMIL$ for some $\seq{x,w} \in P$, then no extension of $q$ can force that $\lim_y \Phi^{\vec{X}}(x,y) = j$. This means that $q$ itself forces that $\lim_y \Phi^{\vec{X}}(x,y) \simeq 1-j$. But notice that for every $e$, the set
	\[
		W_e = \{ q \in \mathbb{C} : \exists \seq{x,w} \in P~(w \geq e \wedge \sigma^q_0(w) =  \RADOMIL) \}
	\]
	is dense in $\mathbb{C}$, and that it is $\Sigma^0_1$-definable from $P$. In this case, we claim that $\vec{X} \oplus P$ will compute an infinite homogeneous set for $\Phi^{\vec{X}}$, and so we let $Y_{s+1} = Y_s \cup \{P\}$. To prove the claim, note that since the sequence of conditions $p_0 \geq p_1 \geq \cdots$ will be generic relative to $P$, for each $e$ there will be some condition in this sequence meeting $W_e$. Thus, for each $e$ there will be an $\seq{x,w} \in P$ with $w \geq e$ such that $A_0(w) =  \RADOMIL$. This means that to compute an infinite limit-homogeneous set for $\Phi^{\vec{X}}$ (and from there an infinite homogeneous one), $\vec{X} \oplus P$ will only need to search for pairs $\seq{x,w} \in P$ with $A_0(w) =  \RADOMIL$, assured that these can be found for arbitrarily large $w$, and that $\lim_y \Phi^{\vec{X}}(x,y)$ must be $1-j$.

	Case 2: otherwise. Let $P$ be the set of all $x$ equal to $\beta(n)$ for some $\beta \in S$. If there is an extension of $q_n$ that, for some $x \in P$, forces that $\lim_y \Phi^{\vec{X}}(x,y) = j$, we let $q_{n+1}$ be this extension, and let $\alpha_{n+1}$ be any $\beta \in S$ with $\beta(n) = x$. If, on the other hand, there is no extension of $q_n$ as above, then $P$ will be an infinite limit-homogeneous set for $\Phi^{\vec{X}}$ with color $1-j$. Hence, $\vec{X} \oplus P$ will compute an infinite homogeneous set for $\Phi^{\vec{X}}$. In this case, we let $Y_{s+1} = Y_s \cup \{P\}$.

	We complete this stage of the construction as follows. If we ended up adding some $P$ to $Y$, we do nothing more. Otherwise, we must have succeeded in defining $\alpha_{n+1}$ from every non-terminal $\alpha_n$. But since $T_0$, being a subtree of $T$, is well-founded, some $\alpha_n$ must be terminal in $T_0$. Choose any $F \subseteq \alpha_n$ such that $\Psi^{H^\Phi_{j,s} \cup F}(w) \downarrow = 1$ for some $w \geq k$, say with use $u$. Let $p_{s+1} = q_n$, $H^\Phi_{j,s+1} = H^\Phi_{j,s} \cup F$, and $I_{s+1} = \{x \in I_s : x > u\}$.

	\medskip
	\emph{Verification.} It is clear from the construction that the $\mathcal{P}$ and $\mathcal{Q}$ requirements are all satisfied. For the $\mathcal{R}$ requirements, suppose that $\Phi^{\vec{R}}$ has no $(\vec{X} \oplus P)$-computable infinite homogeneous set for any $P \in Y$, and fix $\Psi$ and $j$. Then at every stage dedicated to $\mathcal{R}_{\Phi,\Psi,j}$ we succeed in adding one more element $w$ to the set defined by $H^\Phi_j$, whilst ensuring that $A_0(w)$ is $0$ or $1$, depending as the number of previous such stages was even or odd. Hence, $H^\Phi_j$ intersects both $A_0$ and $\overline{A_0}$ infinitely often, as needed. This completes the proof.
\end{proof}

\begin{corollary}\label{C:COHnscredSRT22}
$\COH \nscred \SRT^2_2$.
\end{corollary}

We do not know how to extend our technique to show more generally that $\COH \nscred \SRT^2_{<\infty}$, or even that $\COH \nscred \SRT^2_k$ for any $k \geq 2$. The reason is that when more colors are involved, it no longer follows that if no element of some set limits to a given color, then all the elements of that set limit to a given other color. Thus, adding such a set to the collection $Y$ above does not appear to produce a homogeneous set. Now it is not difficult to generalize our argument to get around this problem when adding a set to $Y$ at a stage dedicated to a $\mathcal{Q}$ requirement, or in Case 2 of a stage dedicated to a $\mathcal{R}$ requirement. But Case 1 seems more involved.

\section{Summary and questions}

We summarize the principal consequences of our results, and how they fit in with prior known ones, in Figures~\ref{F:uniform} and~\ref{F:strong}, the first for the case of $\ured$ and the second for the case of $\scred$. Here, an arrow from $\mathsf{P}$ to $\mathsf{Q}$ means that $\mathsf{Q}$ is reducible to $\mathsf{P}$.

We conclude with some of the questions left over from, and raised by, our work. We have already mentioned these in the text above, but we collect them together here for convenience.

\begin{question}
	Is it the case that $\COH$ is generalized uniformly reducible to $\SRT^2_{<\infty}$? (By~\cite{HJ-TA}, Propositions 4.7 and 4.8, we can replace $\SRT^2_{<\infty}$ here by $\D^2_2$.)
\end{question}

\begin{question}
	For $j < k$, is it the case that $\RT^1_k \ured \SRT^2_j$?
\end{question}

\begin{question}
	For $k > 2$, is it the case that $\COH \scred \SRT^2_k$? Is it the case that $\COH \scred \SRT^2_{<\infty}$?
\end{question}

\noindent Finally, we would like to know more about the complexity of the construction in Section~\ref{S:COH_strong}. Hirschfeldt and Jockusch~\cite[Theorem 3.9]{HJ-TA} gave another argument about $\scred$ involving higher levels of the hyperarithmetical hierarchy, that Patey~\cite[Theorem 3.2]{Patey-TA} obtained independently using a $\Delta^0_2$ construction. Our argument is very different from the one in~\cite{HJ-TA}, but we can ask the same question about whether the repeated use of hyperjumps there is really necessary.

\begin{question}
	Can the family $\vec{X}$ constructed in the proof of Theorem~\ref{T:nscred} be chosen to be arithmetical, or at least hyperarithmetical?
\end{question}

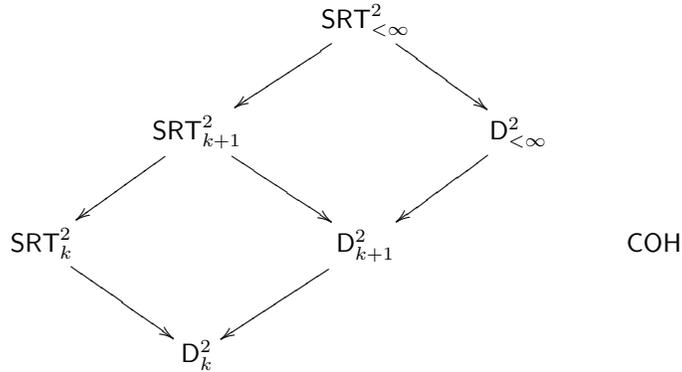
\begin{figure}
\[
\xymatrix{
& & \SRT^2_{<\infty} \ar[dl] \ar[dr] \\
& \SRT^2_{k+1} \ar[dl] \ar[dr] & & \D^2_{<\infty} \ar[dl]\\
\SRT^2_k \ar[dr] & & \D^2_{k+1} \ar[dl] & & \COH\\
& \D^2_k
}
\]
\caption{Relationships between $\SRT^2$, $\D^2$, and $\COH$ under $\ured$, with $k \geq 2$ arbitrary. All the implications here are straightforward, and no reductions hold other than the ones shown, which is justified as follows. That the arrows from $\SRT^2_{<\infty}$ to $\D^2_{<\infty}$, from $\SRT^2_{k+1}$ to $\D^2_{k+1}$, and from $\SRT^2_k$ to $\D^2_k$ cannot be reversed is by Corollary~\ref{C:SRT22nuredD2}. That there are no arrows from $\D^2_k$ to $\D^2_{k+1}$, from $\D^2_{k+1}$ to $\D^2_{<\infty}$, from $\SRT^2_k$ to $\SRT^2_{k+1}$, or from $\SRT^2_{k+1}$ to $\SRT^2_{<\infty}$ follows by a result of Patey~\cite[Corollary 3.6]{Patey-TA}. Finally, that there are no arrows from $\COH$ to any of these principles follows by results of Hirschfeldt et al.~\cite[Theorems 2.3 and 3.7]{HJKLS-2008}, while that no arrows go the other way is by Corollary~\ref{C:COHnuredSRT2}.}\label{F:uniform}
\end{figure}

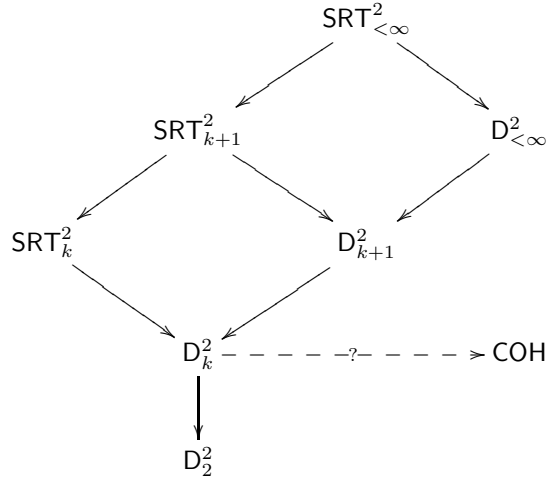
\begin{figure}
\[
\xymatrix{
& & \SRT^2_{<\infty} \ar[dl] \ar[dr] \\
& \SRT^2_{k+1} \ar[dl] \ar[dr] & & \D^2_{<\infty} \ar[dl]\\
\SRT^2_k \ar[dr] & & \D^2_{k+1} \ar[dl] \\
& \D^2_k \ar[d] \ar@{-->}[rr]|-{?} & & \COH\\
& \D^2_2
}
\]
\caption[]{Relationships between $\SRT^2$, $\D^2$, and $\COH$ under $\ured$, with $k \geq 3$ arbitrary. No reductions hold other than the ones shown, except possibly from $\D^2_j$ to $\COH$ for some $j \geq 3$. That the arrows from $\SRT^2_{<\infty}$ to $\D^2_{<\infty}$, from $\SRT^2_{k+1}$ to $\D^2_{k+1}$, and from $\SRT^2_k$ to $\D^2_k$ cannot be reversed is by Corollary~\ref{C:SRT22nscredD2}. That there is no arrow from $\D^2_2$ to $\COH$ is Corollary~\ref{C:COHnscredSRT22}. The other non-reducibilities are justified as in the uniform case.}\label{F:strong}
\end{figure}


\bibliography{/Users/damir/Documents/Papers/Papers.bib}
\bibliographystyle{/Users/damir/Dropbox/Latex/Templates/Files/simple.bst}

\end{document}